\documentclass[a4paper,oneside,reqno]{amsart}
\usepackage{amsthm}
\usepackage{amstext}
\usepackage{amssymb}
\usepackage{graphicx}
\usepackage{esint}

\makeatletter

\numberwithin{equation}{section}
\numberwithin{figure}{section}

\theoremstyle{plain}
\newtheorem{theorem}{Theorem}[section]
\newtheorem{lemma}[theorem]{Lemma}
\theoremstyle{remark}
\newtheorem{rem}[theorem]{Remark}
\newenvironment{remark}{\begin{rem}}{\hfill$\lozenge$
\end{rem}}

\DeclareMathOperator\Real{Re}

\renewcommand{\Re}{\Real}

\DeclareMathOperator\dom{dom}

\DeclareMathOperator\supp{supp}

\DeclareMathOperator\sign{sign}

\newcommand{\ov}{\overline}
\newcommand{\ds}{\displaystyle}

\newcommand{\rd}{\,\mathrm{d}}

\newcommand{\cF}{\mathcal F}

\newcommand{\cL}{\mathcal L}

\newcommand{\CC}{\mathbb C}
\newcommand{\NN}{\mathbb N}

\newcommand{\RR}{\mathbb R}
\newcommand{\ZZ}{\mathbb Z}

\newcommand{\frl}{\mathfrak{l}}

\newcommand{\defeq}{\mathrel{\mathop:}=}

\newcounter{counter_a}
\newenvironment{myenum}{\begin{list}{{\rm(\roman{counter_a})}}%
{\usecounter{counter_a}
\setlength{\itemsep}{0.5ex}\setlength{\topsep}{0.7ex}
\setlength{\leftmargin}{5ex}\setlength{\labelwidth}{5ex}}}{\end{list}}

\newcounter{counter_steps}

\newcommand{\tell}{\tilde{\ell}}
\newcommand{\tellL}[1]{\tilde{\ell}_{{\rm L},#1}}
\newcommand{\tellF}[1]{\tilde{\ell}_{{\rm F},#1}}
\newcommand{\tellM}[1]{\tilde{\ell}_{{\rm M},#1}}
\newcommand{\Li}{\operatorname{Li}}
\newcommand{\rmO}{\mathrm{O}}
\newcommand{\rmo}{\mathrm{o}}
\newcommand{\FWHM}{\operatorname{FWHM}}

\newcommand{\Lkmin}{L_{k,\rm min}}
\newcommand{\Lkmax}{L_{k,\rm max}}
\newcommand{\tLL}[1]{\widetilde{L}_{{\rm L},#1}}
\newcommand{\tLF}[1]{\widetilde{L}_{{\rm F},#1}}
\newcommand{\tLM}[1]{\widetilde{L}_{{\rm M},#1}}
\newcommand{\tLT}{\widetilde{L}_T}
\newcommand{\deltakmax}{\delta_{k,\rm max}}
\newcommand{\deltaonemax}{\delta_{1,\rm max}}

\makeatother

\begin{document}

\title[Path Laplacian operators and superdiffusive processes]{Path Laplacian
operators and \\ superdiffusive processes on graphs. I. \\ One-dimensional case}

\author{Ernesto Estrada, Ehsan Hameed, Naomichi Hatano, Matthias Langer}

\address{Department of Mathematics \& Statistics, University of Strathclyde, \newline
26~Richmond Street, Glasgow G1~1XH, UK}
\address{Institute of Industrial Science, University of Tokyo,
4-6-1 Komaba, Meguro, Tokyo 153-8505, Japan}

\begin{abstract}
We consider a generalization of the diffusion equation on graphs.
This generalized diffusion equation gives rise to both normal and superdiffusive
processes on infinite one-dimensional graphs. The generalization is
based on the $k$-path Laplacian operators $L_{k}$, which account
for the hop of a diffusive particle to non-nearest neighbours in a
graph. We first prove that the $k$-path Laplacian operators are self-adjoint.
Then, we study the transformed $k$-path Laplacian operators using
Laplace, factorial and Mellin transforms. We prove that the generalized
diffusion equation using the Laplace- and factorial-transformed operators
always produce normal diffusive processes independently of the parameters
of the transforms. More importantly, the generalized diffusion equation
using the Mellin-transformed $k$-path Laplacians $\sum_{k=1}^{\infty}k^{-s}L_{k}$
produces superdiffusive processes when $1<s<3$.
\\[1ex]
\textit{2010 Mathematics Subject Classification:} 47B39; 47B25, 60J60, 05C99
\\[0.5ex]
\textit{Keywords:} $k$-path Laplacian, anomalous diffusion
\end{abstract}

\maketitle

\section{Introduction}

\noindent
Superdiffusive processes are ubiquitous in many natural systems, ranging
from physical to biological and man-made ones. They refer to those
anomalous diffusive processes where the mean square displacement (MSD)
of the diffusive particle scales nonlinearly with time.  We refer the
reader to \cite{Anomalous diffison review} and the references therein
for the background and applications of anomalous diffusion.  The superdiffusive
processes have been modelled in many different ways (see \cite{Anomalous diffison review}
for a review and analysis). The most used models, however, are based
on random walks with L\'evy flights (RWLF) \cite{Anomalous_Levy flights}
and on the use of the fractional diffusion equation (FDE) \cite[Chapter~11]{Fractional Models}.
There are different types of definitions of fractional derivative,
such as the Caputo fractional operator and the Riemann--Liouville
fractional operator \cite{Podlubny}, which then have different interpretations
and adapt differently to the different physical phenomena studied
with them (see \cite{fractional diffusion_1,fractional diffusion_2}).

Recently, anomalous diffusion of ultracold atoms has been observed in
a discrete one-dimensional system \cite{cold_atoms}.
The model considered in that work for explaining the superdiffusive process
is a simple diffusion model in which the particles are located in real space,
each having a velocity which fluctuates in time due to interaction with a bath.
Then, after some time the particles' position is distributed
in a non-Gaussian way and the full width at half maximum (FWHM)
scales as a power-law of the time with a signature characteristic
of superdiffusion.  The mathematical framework used to describe this
anomalous diffusion was based on the FDE.  However, an alternative
view of this process is possible.  First, we can consider that the
diffusive particle is diffusing in a one-dimensional discrete space.
Then, we can consider that the diffusive particle is not only hopping
to its nearest neighbours in the 1D lattice, but to any other point
of it with a probability that scales with the distance between the
two places.  In the current work we prove analytically that such kind
of processes can give rise to superdiffusion under certain conditions.
We should remark that existence of such long-range hops in diffusive process
has been well documented since the 1990s on experimental basis of
different nature.  First, the group of G.~Ehrlich
\cite{senft1995long} observed experimentally significant contributions
to the thermodynamical properties of the self-diffusion of weakly
bounded Pd atoms from jumps spanning second and third nearest-neighbours
in the metallic surface.  Since then, the role
of long jumps in adatom and admolecules diffusing on metallic surfaces
has been confirmed in many different systems \cite{yu2013single,ala2002collective}.

The study of diffusion on graphs is a well-established physico-mathematical theory based
on the graph-theoretic version of the diffusion equation
\cite{Diffusion on graphs,Diffusion on graphs 2}
\begin{align}
  \frac{\rd}{\rd t}u(t) & =-Lu(t),\label{eq:diff_1}\\[1ex]
  u(0) & =u_{0},\label{eq:diff_2}
\end{align}
where $L$ --- the discrete Laplacian --- is defined via the adjacency
matrix $A$ of the graph and the diagonal matrix of vertex degrees
$K$ as $L=K-A$ \cite{Mugnolo book,Laplacian_1,Laplacian_2}. The
Laplacian matrix has been extended to infinite, connected and locally
finite graphs and studied as an operator in the Hilbert space $\ell^2$ over the vertices
\cite{Laplacian operator_1,Laplacian operator_2,jorgensen08,Laplcian operator_4,
Laplacian operator_5,weber10,wojciechowski_PhD_thesis,HKLW12}.
Although RWLF and the FDE have been applied to study diffusion on
graphs (see for instance \cite{Levy flights graphs,fractional diffusion graphs}),
the question that arises here is whether is it possible to design
a simple graph-theoretic, physically sound and mathematically elegant
method based on a generalization of the Laplacian operator in \eqref{eq:diff_1}
to account for the superdiffusive process observed in physical phenomena.
An appropriate scenario for this generalization is to consider that
the diffusive particle can hop not only to its nearest neighbours --- as
controlled by $L$ in \eqref{eq:diff_1} --- but to any other node of
the graph, with a probability that decays with the increase of the
shortest path distance separating the node in which the particle is
currently located to the one to which it will hop. A generalization
of the Laplacian matrix --- known as the $k$-path Laplacian --- that
takes into account such long-range hops of the diffusive particle
has been recently considered for finite undirected graphs \cite{path Laplacians}.

The aim of this article is twofold. First, we extend the $k$-path
Laplacians $L_{k}$ \cite{path Laplacians} to consider connected
and locally finite infinite graphs. We prove here that these operators
are self-adjoint. We also study the transformed $k$-path Laplacian
operators using Laplace, factorial and Mellin transforms. We then
study an infinite linear chain and obtain analytical expressions for
the transformed $k$-path Laplacians operators as well as for the
exponential operators of both, the $k$-path Laplacians and their
transformations. Second, we plug this generalized Laplacian operators
into the graph-theoretic diffusion equation (\ref{eq:diff_1}--\ref{eq:diff_2})
to obtain a generalized diffusion equation for graphs. We prove that
when the Laplace- and factorial-transformed operators are used in
the generalized diffusion equation, the diffusive processes observed
are always normal independently of the parameters of the transforms.
For the Mellin-transformed $k$-path Laplacians $\sum_{k=1}^{\infty}k^{-s}L_{k}$
we find that the diffusion is normal only when $s>3$.
When $1<s<3$, however, the time evolution is superdiffusive with the superdiffusive
exponent being $\kappa=\frac{2}{s-1}$, which leads to arbitrary values
for $\kappa$ in $(1,\infty)$. We remind that in general we can find
that ${\rm MSD}\sim t^{\kappa}$, where the diffusion is \textit{normal}
when $\kappa=1$, while it is a superdiffusive process when $\kappa>1$.
The particular case when $\kappa=2$ is known as \textit{ballistic}
diffusion, which is characterized by the fact that at small times
the particles are not hindered yet by collisions and diffuse very fast.
In a follow-up paper (Part II) we shall study the two-dimensional situation.

\section{The $k$-path Laplacian operators}

\noindent
In this work we always consider $\Gamma=(V,E)$ to
be an undirected finite or infinite graph with vertices $V$ and edges $E$.
We assume that $\Gamma$ is connected and locally finite
(i.e.\ each vertex has only finitely many edges emanating from it).
Let $d$ be the distance metric on $\Gamma$,
i.e.\ $d(v,w)$ is the length of the shortest path from $v$ to $w$,
and let $\delta_{k}(v)$ be the $k$-path degree of the vertex $v$,
i.e.\
\begin{align}
\delta_{k}(v)\defeq\#\{w\in V:d(v,w)=k\}.
\end{align}
Since $\Gamma$ is locally finite, $\delta_{k}(v)$ is finite for
every $v\in V$. Denote by $C(V)$ the set of all complex-valued functions
on $V$ and by $C_{0}(V)$ the set of complex-valued functions on
$V$ with finite support. Moreover, let $\ell^{2}(V)$ be the Hilbert
space of square-summable functions on $V$ with inner product
\begin{align}
\langle f,g\rangle=\sum_{v\in V}f(v)\ov{g(v)},\qquad f,g\in\ell^{2}(V).
\end{align}
In $\ell^{2}(V)$ there is a standard orthonormal basis consisting
of the vectors $e_{v}$, $v\in V$, where
\begin{equation}\label{defev}
  e_{v}(w)\defeq\begin{cases}
    1 & \text{if \ensuremath{w=v}},\\[0.5ex]
    0 & \text{otherwise}.
\end{cases}
\end{equation}

Let $\cL_{k}$ be the following mapping from $C(V)$ into itself:
\begin{equation}\label{eq:path_Laplacian}
  \bigl(\cL_{k}f\bigr)(v)\defeq\sum_{w\in V:\,d(v,w)=k}\bigl(f(v)-f(w)\bigr),\qquad f\in C(V).
\end{equation}
This means that by replacing $L$ in \eqref{eq:diff_1} by $\cL_k$
in \eqref{eq:path_Laplacian}
we obtain a diffusive process in which the diffusive particle hops
to nodes which are separated by $k$ edges from its current location.
This represents a natural extension of the idea of diffusion on graphs
where the particle can only hops to nearest neighbours from its current
position. As it has been analysed in \cite{path Laplacians}, the
so-called $k$-path Laplacian naturally extends the concept of graph
connectivity, i.e.\ whether a graph is connected or not, to the $k$-connectivity,
which indicates whether every node in the graph can be reached by
a particle which is $k$-hopping from node to node in the graph.

On the vectors $e_v$ it acts as follows:
\begin{equation}\label{Lkev}
  (\cL_k e_v)(w) = \begin{cases}
    \delta_{k}(v) & \text{if \ensuremath{w=v}},\\[0.5ex]
    -1 & \text{if \ensuremath{d(v,w)=k}},\\[0.5ex]
    0 & \text{otherwise}.
  \end{cases}
\end{equation}

We define $\Lkmin$ and $\Lkmax$,
the \emph{minimal and maximal $k$-path Laplacians}, as the restrictions of $\cL_k$ to
\[
  \dom(\Lkmin) = C_0(V) \qquad\text{and}\qquad
  \dom(\Lkmax) = \bigl\{ f\in\ell^{2}(V):\cL_{k}f\in\ell^{2}(V)\bigr\},
\]
respectively.
Clearly, $e_v\in\dom(\Lkmin)$, and we obtain from \eqref{Lkev} that
\begin{equation}
  \bigl\|\Lkmin e_{v}\bigr\|
  = \sqrt{\bigl(\delta_{k}(v)\bigr)^{2}+\delta_{k}(v)}
  = \begin{cases}
    0 & \text{if \ensuremath{\delta_{k}(v)=0}},\\[2ex]
    \ds\delta_{k}(v)\sqrt{1+\frac{1}{\delta_{k}(v)}}\quad &
    \text{if \ensuremath{\delta_{k}(v)>0}}.
  \end{cases}\label{normev}
\end{equation}
First we show that $\Lkmin^*=\Lkmax$.
To this end let $f\in C_0(V)$ and $g\in C(V)$, let $V_{00}$ be the support of $f$ and set
\begin{equation}\label{defV0}
  V_{0} \defeq V_{00}\cup\bigl\{ v\in V:\exists\,w\in V_{00}\;\;
  \text{such that}\;\;d(v,w)=k\bigr\},
\end{equation}
which is a finite set.
Then $\supp\cL_k f\subset V_0$ and the following relation holds:
\begin{align}
  & \sum_{v\in V}(\cL_k f)(v)\ov{g(v)}
  = \sum_{v\in V_0}(\cL_k f)(v)\ov{g(v)}
  = \sum_{\substack{v,w\in V_0: \\[0.3ex] d(v,w)=k}}
  \bigl(f(v)-f(w)\bigr)\ov{g(v)}
  \notag\\[1ex]
  &= \frac{1}{2}\left[\,\sum_{\substack{v,w\in V_{0}: \\[0.3ex] d(v,w)=k}}
  \bigl(f(v)-f(w)\bigr)\ov{g(v)}+\sum_{\substack{v,w\in V_{0}: \\[0.3ex] d(v,w)=k}}
  \bigl(f(w)-f(v)\bigr)\ov{g(w)}\right]
  \notag\\[1ex]
  &= \frac{1}{2}\sum_{\substack{v,w\in V_{0}: \\[0.3ex] d(v,w)=k}}
  \bigl(f(v)-f(w)\bigr)\ov{\bigl(g(v)-g(w)\bigr)}
  \label{formfg} \displaybreak[0]\\[1ex]
  &= \frac{1}{2}\left[\,\sum_{\substack{v,w\in V_{0}: \\[0.3ex] d(v,w)=k}}
  f(v)\ov{\bigl(g(v)-g(w)\bigr)}+\sum_{\substack{v,w\in V_{0}: \\[0.3ex] d(v,w)=k}}
  f(w)\ov{\bigl(g(w)-g(v)\bigr)}\right]
  \notag\displaybreak[0]\\[1ex]
  &= \sum_{\substack{v,w\in V_{0}: \\[0.3ex] d(v,w)=k}}f(v)\ov{\bigl(g(v)-g(w)\bigr)}
  = \sum_{v\in V_{00}}f(v)\ov{(\cL_k g)(v)}
  \notag\\[1ex]
  &= \sum_{v\in V}f(v)\ov{(\cL_k g)(v)}.
  \label{Lkadj}
\end{align}
Let $g\in\dom(\Lkmax)$.
It follows from \eqref{Lkadj} that
\[
  \langle \Lkmin f,g\rangle = \langle f,\Lkmax g\rangle
\]
for all $f\in\dom(\Lkmin)$, which implies that $g\in\dom(\Lkmin^*)$.
Now let $g\in\dom(\Lkmin^*)$.
For each $v\in V$ we obtain from \eqref{Lkadj} with $f=e_v$ that
\begin{align*}
  \ov{(\Lkmin^*g)(v)} &= \langle e_v,\Lkmin^*g\rangle
  = \langle \Lkmin e_v,g\rangle
  = \sum_{w\in V}(\cL_k e_v)(w)\ov{g(w)}
  \\[1ex]
  &= \sum_{w\in V}e_v(w)\ov{(\cL_k g)(w)} = \ov{(\cL_k g)(v)},
\end{align*}
which implies that $\Lkmin^*g=\cL_k g$.
Since $\Lkmin^*g\in\ell^2(V)$ by the definition of the adjoint, it follows
that $g\in\dom(\Lkmax)$.  Hence $\Lkmin^*=\Lkmax$.

Since $\Lkmax$ is an extension of $\Lkmin$, it follows that $\Lkmin$
is a symmetric operator.
Moreover, for $f=g$ we obtain from \eqref{formfg} that
\begin{equation}
  \bigl\langle\Lkmin f,f\bigr\rangle
  = \frac{1}{2}\sum_{\substack{v,w\in V_{0}: \\[0.3ex] d(v,w)=k}}
  \bigl|f(v)-f(w)\bigr|^{2},
  \label{formf}
\end{equation}
where $V_0$ is as in \eqref{defV0};
this shows that $\Lkmin$ is a non-negative operator.

We say that a subset $V_{0}$ of $V$ is \emph{$k$-connected} if
each pair $v,w\in V_{0}$ is connected by a $k$-hopping walk. The
set $V_{0}\subset V$ is called a \emph{$k$-connected component} of $V$ if
$V_{0}$ is a maximal $k$-hopping connected subset of $V$. If $V_{0}\subset V$
is a $k$-hopping component, then $C(V_{0})$ considered as a subspace
of $C(V)$ is $\cL_{k}$-invariant.

\begin{lemma}\label{le:maxprincple}
Let $V_{0}$ be a $k$-connected component
of\, $V$ and let $f\in C(V_{0})$ be real-valued and bounded such
that $f$ attains its supremum. If
\begin{equation}\label{nonpos}
  \bigl(\cL_{k}f\bigr)(v)\le0\quad\text{for every }v\in V_{0},
\end{equation}
then $f$ is constant on $V_{0}$.
\end{lemma}

\begin{proof}
Assume that $f$ is not constant. Then there exist $v_{0},v_{1}\in V_{0}$
such that
\begin{align*}
 & f(v_{0})=\max\{f(v):v\in V_{0}\},\\[1ex]
 & f(v_{1})<f(v_{0}),\qquad d(v_{1},v_{0})=k.
\end{align*}
This implies that
\[
  \bigl(\cL_{k}f\bigr)(v_{0}) = f(v_{0})-f(v_{1})+\sum_{\substack{w\ne v_{1}:\\[0.3ex]
  d(w,v_{0})=k}}\bigl(f(v_{0})-f(w)\bigr)
  > 0,
\]
which is a contradiction to \eqref{nonpos}.  Hence $f$ is constant on $V_{0}$.
\end{proof}

Next we show that $\Lkmin$ is actually essentially self-adjoint;
see, e.g.\ \cite{jorgensen08,weber10,wojciechowski_PhD_thesis} for the case $k=1$.

\begin{theorem}
The operator $\Lkmin$ is essentially self-adjoint and hence $\Lkmax$
is equal to the closure of $\Lkmin$.
\end{theorem}

\begin{proof}
Since $\Lkmin$ is non-negative and $\Lkmin^{*}=\Lkmax$, it is sufficient
to show that $-1$ is not an eigenvalue of $\Lkmax$. Assume that
this is not the case. Then there exists an $f\in\ell^{2}(V)$ such
that $f\not\equiv0$ and $\Lkmax f=-f$. The function $f$ must be
zero on every finite $k$-hopping component since $\Lkmax$ restricted
to such a component is self-adjoint and non-negative. Therefore there
exists an infinite $k$-hopping component $V_{0}$ where $f$ is not
identically zero. It follows that
\[
  \delta_{k}(v)f(v)-\sum_{w:\,d(v,w)=k}f(w)=-f(v)
\]
for $v\in V_{0}$, or equivalently,
\[
  \bigl(\delta_{k}(v)+1\bigr)f(v)=\sum_{w:\,d(v,w)=k}f(w).
\]
Taking the modulus on both sides we obtain
\[
  \bigl(\delta_{k}(v)+1\bigr)|f(v)|\le\sum_{w:\,d(v,w)=k}|f(w)|.
\]
Now we consider the function $|f|$:
\[
  \bigl(\cL_{k}|f|\bigr)(v)=\delta_{k}(v)|f(v)|-\sum_{w:\,d(v,w)=k}|f(w)|\le-|f(v)|\le0.
\]
Since $f|_{V_{0}}\in\ell^{2}(V_{0})$, the function $|f|$ attains
the supremum on $V_{0}$. Hence Lemma~\ref{le:maxprincple} yields
that $|f|$ is constant on $V_{0}$. This implies that $f=0$ on $V_{0}$
because $V_{0}$ is infinite; a contradiction.
\end{proof}

We denote the closure of $\Lkmin$ by $L_{k}$ and call
it the \emph{$k$-path Laplacian}. By the previous theorem we have
$L_{k}=\Lkmax$; it is a self-adjoint and non-negative operator in
$\ell^{2}(V)$.
Note the difference in notation between the mapping $\cL_k$ acting in $C(V)$
and the self-adjoint operator $L_k$ in $\ell^2(V)$.

We can now estimate forms: for $f\in\dom(\Lkmin)=C_{0}(V)$ we obtain
from \eqref{formf} that
\begin{align}
  \bigl\langle\Lkmin f,f\bigr\rangle
  &= \frac{1}{2}\sum_{\substack{v,w\in V: \\[0.3ex] d(v,w)=k}}
  \bigl|f(v)-f(w)\bigr|^{2}
  \le \frac{1}{2}\sum_{\substack{v,w\in V: \\[0.3ex] d(v,w)=k}}
  \Bigl(|f(v)|+|f(w)|\Bigr)^{2}
  \notag\\[1ex]
  &\le \sum_{\substack{v,w\in V: \\[0.3ex] d(v,w)=k}}
  \Bigl(|f(v)|^{2}+|f(w)|^{2}\Bigr)
  \notag\\[1ex]
  &= \sum_{v\in V}\delta_{k}(v)|f(v)|^{2} + \sum_{w\in V}\delta_{k}(w)|f(w)|^{2}
  \notag\\[1ex]
  &= 2\sum_{v\in V}\delta_{k}(v)|f(v)|^{2}.
  \label{formfest}
\end{align}

In the next theorem we answer the question when $L_k$ is a bounded operator.

\begin{theorem}\label{th:Lkbounded}
The operator $L_{k}$ is bounded if and only if\,
$\delta_{k}$ is a bounded function on $V$.
Now assume that $\delta_{k}$ is bounded and set
\begin{equation}
  \deltakmax \defeq \max\{\delta_{k}(v):v\in V\};
  \label{defdeltakmax}
\end{equation}
then
\begin{equation}
\deltakmax\le\|L_{k}\|\le2\deltakmax.\label{boundnorm}
\end{equation}
\end{theorem}

\begin{proof}
If $\delta_{k}$ is unbounded, then \eqref{normev} immediately shows
that $L_{k}$ is unbounded. Now assume that $\delta_{k}$ is bounded.
Relation \eqref{normev} yields the lower bound for $\|L_{k}\|$ in \eqref{boundnorm}.
From \eqref{formfest} we obtain that for $f\in\dom(\Lkmin)$,
\[
  \bigl\langle\Lkmin f,f\bigr\rangle\le2\deltakmax\sum_{v\in V}|f(v)|^{2}=2\deltakmax\|f\|^{2}.
\]
Since $L_{k}$ is self-adjoint and $L_{k}$ is the closure of $\Lkmin$,
this shows that $L_{k}$ is bounded and that $\|L_{k}\|\le2\deltakmax$.
\end{proof}

\section{Transformed $k$-path Laplacian operators}

\noindent
We consider series of the form
\begin{equation}
  \sum_{k=1}^{\infty}c_{k}L_{k}
  \label{seriesgen}
\end{equation}
with $c_{k}\in\CC$.  If all $L_{k}$ are bounded and
\begin{equation}\label{convcond}
  \sum_{k=1}^{\infty}|c_{k}|\,\|L_{k}\|<\infty,
\end{equation}
then the series in \eqref{seriesgen} converges to a bounded operator on $\ell^2(V)$.
If, in addition, $c_k\in\RR$ for all $k\in\NN$, then the operator in \eqref{seriesgen}
is self-adjoint; if $c_k\ge0$ for all $k\in\NN$, then it is a non-negative operator.

In the following we discuss three transformed operators in more detail:
the Laplace, the factorial and the Mellin transforms.

\begin{theorem}
Assume that $\delta_1$ is bounded on $V$ and let $\deltaonemax$ be
as in \eqref{defdeltakmax}.
\begin{myenum}
\item
The \emph{Laplace-transformed $k$-Laplacian}
\begin{equation}\label{deftLL}
  \tLL{\lambda} \defeq \sum_{k=1}^\infty e^{-\lambda k}L_k
\end{equation}
is a bounded operator when $\lambda\in\CC$ with $\Re\lambda>\ln\deltaonemax$.
It is non-negative if $\lambda\in(\ln\deltaonemax,\infty)$.
\item
The \emph{factorial-transformed $k$-Laplacian}
\begin{equation}\label{deftLF}
  \tLF{z} \defeq \sum_{k=1}^\infty \frac{z^k}{k!}L_k
\end{equation}
is a bounded operator for every $z\in\CC$.
It is self-adjoint if $z\in\RR$ and non-negative if $z\ge0$.
\item
Assume that $\deltakmax$ satisfies
\begin{equation}\label{deltapolyn}
  \deltakmax \le Ck^\alpha
\end{equation}
for some $\alpha\ge0$ and $C>0$; then the \emph{Mellin-transformed $k$-Laplacian}
\begin{equation}\label{deftLM}
  \tLM{s} \defeq \sum_{k=1}^\infty \frac{1}{k^s}L_k
\end{equation}
is a bounded operator for $s\in\CC$ with $\Re s>\alpha+1$.

Under the assumption \eqref{deltapolyn} the operator $\tLL{\lambda}$
from \eqref{deftLL} is bounded for every $\lambda\in\CC$ with $\Re\lambda>0$.
\end{myenum}
\end{theorem}

\begin{proof}
It follows easily that $\deltakmax \le \deltaonemax^k$ and hence
\[
  \|L_k\| \le 2\deltaonemax^k
\]
for every $k\in\NN$ by Theorem~\ref{th:Lkbounded}.
Therefore the convergence condition \eqref{convcond} is satisfied in items (i) and (ii)
for the specified $\lambda$ and $z$.

For item (iii) we observe that under the condition \eqref{deltapolyn}
the operators $L_k$ satisfy
\[
  \|L_k\| \le 2Ck^\alpha.
\]
Hence also in this case the condition \eqref{convcond} is satisfied for $\tLM{s}$
with $\Re s>\alpha+1$ and for $\tLL{\lambda}$ with $\Re\lambda>0$.
\end{proof}

If the graph is finite, then there is no restriction on the parameters needed,
i.e.\ one can choose any $\lambda\in\CC$ in (i) and any $s\in\CC$ in (iii).

The growth condition \eqref{deltapolyn} is fulfilled for several infinite graphs
such as a linear path graph (or chain) for which $\deltakmax=2$ for every $k\in\NN$,
an infinite ladder for which $\deltakmax=4$,
and for triangular, square and hexagonal lattices
for which $\deltakmax=gk$, with $g=6,4,3$, respectively, among
many others.  However, it is not fulfilled for Cayley trees for which
$\delta_{k,\max}=r\left(r-1\right)^{k-1}$ where $r$ is the degree
of the non-pendant nodes.

Let us now consider the situation when the operators $L_k$ may be unbounded.
The closed quadratic form $\frl_k$ corresponding to $L_k$
in the sense of \cite[\S VI.1.5]{kato} is given by
\[
  \frl_k[f] \defeq \frac{1}{2}\sum_{\substack{v,w\in V: \\[0.3ex] d(v,w)=k}}
  \bigl|f(v)-f(w)\bigr|^2
\]
with domain
\[
  \dom(\frl_k) = \biggl\{f\in\ell^2(V):
  \sum_{\substack{v,w\in V: \\[0.3ex] d(v,w)=k}}\bigl|f(v)-f(w)\bigr|^2 < \infty\biggr\}.
\]
Assume that $c_k\ge0$, $k\in\NN$.  Then
\begin{equation}\label{sum_of_forms}
  \sum_{k=1}^N c_k\frl_k
\end{equation}
is an increasing sequence of densely defined, closed, non-negative quadratic forms
(see \cite[Theorem~VI.1.31]{kato}).
By \cite[Theorem~VIII.3.13a]{kato} the sequence in \eqref{sum_of_forms}
converges to a closed non-negative quadratic form $\tilde\frl$ that is given by
\begin{align*}
  & \tilde\frl[f] = \sum_{k=1}^\infty c_k\frl_k[f]
  = \frac{1}{2}\sum_{k=1}^\infty c_k \hspace*{-1ex}
  \sum_{\substack{v,w\in V: \\[0.3ex] d(v,w)=k}}\bigl|f(v)-f(w)\bigr|^2, \\[1ex]
  & \dom(\tilde\frl) = \Bigl\{f\in\bigcap_{k=1}^\infty\dom(\frl_k):
  \sum_{k=1}^\infty c_k\frl_k[f]<\infty\Bigr\}.
\end{align*}
Assume now that
\begin{equation}\label{ckdkv}
  \sum_{k=1}^\infty c_k\delta_k(v) < \infty
\end{equation}
for every $v\in V$.  Since
\[
  \frl_k[e_v] = \langle L_k e_v,e_v\rangle = \delta_k(v)
\]
by \eqref{Lkev}, condition \eqref{ckdkv} implies that $e_v\in\dom(\tilde\frl)$
for every $v\in V$, and hence the form $\tilde\frl$
is densely defined.
By \cite[Theorem~VI.2.1]{kato} there exists a self-adjoint non-negative
operator $\tilde L$ that corresponds to $\tilde\frl$ in the sense that
\[
  \tilde\frl[f,g] = \bigl\langle\tilde Lf,g\bigr\rangle
  \qquad \text{for}\;\; f\in\dom\bigl(\tilde L\bigr),\; g\in\dom\bigl(\tilde\frl\bigr).
\]
Moreover, \cite[Theorem~VIII.3.13a]{kato} implies that the partial
sums $\sum_{k=1}^N c_k L_k$ converge in the strong resolvent sense to
the operator $\tilde L$.

As an example consider a tree where each vertex in generation $n\in\NN_0$
has $n+1$ children.  It is easy to see that there are $n!$ vertices
in generation $n$ and that
\[
  \delta_k(v) \le (n+k)!
\]
for each vertex $v$ in generation $n$.
For $z\in(0,1)$ condition \eqref{ckdkv} is satisfied for the factorial transform
since
\[
  \sum_{k=1}^\infty \frac{z^k}{k!}\delta_k(v)
  \le \sum_{k=1}^\infty \frac{z^k}{k!}(n+k)! < \infty
\]
for every vertex $v$ in generation $n$.
Hence $\tLF{z}$ is a self-adjoint operator on this tree.
If one includes linear chains of growing length between each generation,
then $\delta_k(v)$ is growing more slowly and also other transformed $k$-path
Laplacians are self-adjoint operators.

Assume that we are in the situation as above, i.e.\ that $c_k\ge0$ and
that condition \eqref{ckdkv} is satisfied.
It is not difficult to see that the quadratic form $\tilde\frl$ is a Dirichlet form,
i.e.\ it is closed and non-negative and it satisfies $\tilde\frl[Cf]\le\tilde\frl[f]$
for every mapping $C:\CC\to\CC$ with $C(0)=0$ and $|Cx-Cy|\le|x-y|$.
By the Beurling--Deny criteria the operator $-\tilde L$ generates
an analytic, positivity-preserving semigroup of contractions;
see, e.g.\ \cite[Appendix~1 to Section~XIII.12]{reed_simon_IV}.
In the remaining sections we consider a situation where all $L_k$ are bounded operators
and \eqref{convcond} is satisfied.  In this case we can write $(e^{-t\tilde L})_{t\ge0}$
for the semigroup.

\section{The $k$-path Laplacians on the infinite path graph}

\noindent
Let $P_{\infty}$ be the infinite path graph (or chain),
i.e.\ the graph whose vertices can be identified with $\ZZ$ and each pair
of consecutive numbers is connected by a single edge.
We now use index notation and write $u=(u_n)_{n\in\ZZ}$ for elements
in $\ell^2(P_\infty)$.
The $k$-path Laplacian acts as follows
\[
  (L_k u)_n = 2u_n - u_{n+k} - u_{n-k},
  \qquad n\in\ZZ,\; u=(u_\mu)_{\mu\in\ZZ}\in\ell^2(P_\infty).
\]
It can also be identified with a double-infinite matrix
whose entries are
\begin{align}\label{eq3.1}
  (L_{k})_{\mu\nu} = 2\delta_{\mu,\nu}-\delta_{\mu,\nu-k}-\delta_{\mu,\nu+k},
  \qquad \mu,\nu\in\ZZ,
\end{align}
where $\delta$ denotes the Kronecker delta.

In order to consider the diffusion of particles on the graph, we let $e_0$
be as in \eqref{defev}, i.e.\
\begin{equation}
  (e_0)_{n} = \delta_{n,0},
  \label{eq:initial condition}
\end{equation}
which describes a profile that is concentrated at the origin.
Under the application of the standard combinatorial Laplacian $L_{1}$
the particle hops to the neighbouring sites $\pm1$, whereas under the application of
the $k$-path Laplacian $L_{k}$ the particle hops to the sites $\pm k$:
\[
  (L_k e_0)_n = 2\delta_{n,0}-\delta_{n,-k}-\delta_{n,+k}.
\]

Since $\deltakmax=2$ for every $k\in\NN$, the transformed $k$-Laplacians $\tLL{\lambda}$,
$\tLF{z}$ and $\tLM{s}$ from \eqref{deftLL}, \eqref{deftLF} and \eqref{deftLM},
respectively, are bounded operators
for $\lambda\in\CC$ with $\Re\lambda>0$, for every $z\in\CC$ and
every $s\in\CC$ with $\Re s>1$.
These operators are self-adjoint and non-negative if $\lambda\in(0,\infty)$,
$z\in(0,\infty)$ and $s\in(1,\infty)$, respectively.
In the following lemma we find explicit representations of these operators.

\begin{lemma}\label{lem3}
Let $\lambda\in\CC$ with $\Re\lambda>0$, $z\in\CC$ and $s\in\CC$ with $\Re s>1$, and
let $\tLL{\lambda}$, $\tLF{z}$, $\tLM{s}$ be as in \eqref{deftLL}, \eqref{deftLF}
and \eqref{deftLM}, respectively.  Then for any $u\in\ell^2(P_\infty)$ we have
\begin{align*}
  \bigl(\tLL{\lambda}u\bigr)_n &= \frac{2}{e^\lambda-1}u_n
  -\sum_{k=1}^\infty e^{-\lambda k}\bigl(u_{n-k}+u_{n+k}\bigr),
  \\[1ex]
  \bigl(\tLF{z}u\bigr)_n &= 2(e^z-1)u_n
  - \sum_{k=1}^\infty \frac{z^k}{k!}\bigl(u_{n-k}+u_{n+k}\bigr),
  \\[1ex]
  \bigl(\tLM{s}u\bigr)_n &= 2\zeta(s)u_n
  - \sum_{k=1}^\infty \frac{1}{k^s}\bigl(u_{n-k}+u_{n+k}\bigr),
\end{align*}
where $\zeta$ is Riemann's zeta function defined by
\[
  \zeta(s) = \sum_{k=1}^\infty \frac{1}{k^s}\,.
\]
Applying them to $e_0$ we obtain
\begin{gather}
  \bigl(\tLL{\lambda}e_0\bigr)_n = \begin{cases}
    \dfrac{2}{e^\lambda-1} & \text{if}\;\;n=0, \\[2ex]
    e^{-\lambda|n|} & \text{if}\;\;n\ne0,
  \end{cases}
  \hspace*{5ex}
  \bigl(\tLF{z}e_0\bigr)_n = \begin{cases}
    2(e^z-1) & \text{if}\;\;n=0, \\[2ex]
    \dfrac{z^{|n|}}{|n|!} & \text{if}\;\;n\ne0,
  \end{cases}
  \label{tLLFe0}
  \\[1ex]
  \bigl(\tLM{s}e_0\bigr)_n = \begin{cases}
    2\zeta(s) & \text{if}\;\;n=0, \\[2ex]
    \dfrac{1}{|n|^s} & \text{if}\;\;n\ne0.
  \end{cases}
  \label{tLMe0}
\end{gather}
\end{lemma}

\begin{proof}
Let $c_k$, $k\in\NN$, be arbitrary coefficients so that \eqref{convcond} is satisfied.
Then
\[
  \biggl(\sum_{k=1}^\infty c_k L_k u\biggr)_n
  = \biggl(2\sum_{k=1}^\infty c_k\biggr)u_n - \sum_{k=1}^\infty c_k(u_{n-k}+u_{n+k}).
\]
Now the assertions of the lemma follow easily.
\end{proof}

Figure \ref{particle hopping} illustrates the results of Lemma
\ref{lem3} in a graphical form displaying $\tLL{1}e_0$, $\tLF{1}e_0$
and $\tLM{2.5}e_0$ on $21$ nodes.
The plot clearly indicates that the three transforms of the $k$-path
Laplacian operators hop the particles to distant sites in the linear
chain.
\begin{figure}
\centering
\includegraphics[width=0.6\textwidth]{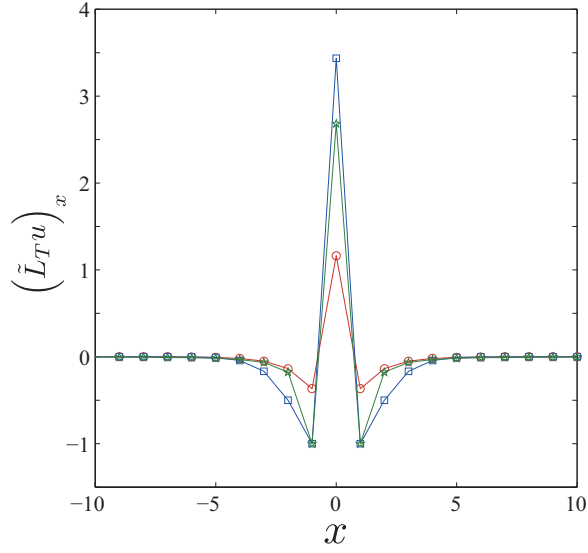}
\caption{Plot of the particle density at the different nodes of a linear path
with 21 nodes obtained from the Laplace (circles), factorial (squares)
and Mellin (stars) transformed $k$-path Laplacians with $\lambda=1$,
$z=1$ and $s=2.5$, respectively.
}
\label{particle hopping}
\end{figure}

\section{Time-evolution operators}

\noindent
Let us now consider the time evolution of the particle density profile
governed by the differential equation
\[
  \frac{\rd}{\rd t}u(t) = -Lu(t)
\]
satisfying the initial equation $u(0)=w$, where $L$ is any of the operators
$L_k$, $\tLL{\lambda}$, $\tLF{z}$ or $\tLM{s}$, where $\lambda\in\CC$ with $\Re\lambda>0$,
$z\in\CC$, $s\in\CC$ with $\Re s>1$ and where $w\in\ell^2(P_\infty)$.
Since $L$ is a bounded operator in all cases, the solution is given by
\begin{equation}\label{time_evolutionL}
  u(t) = e^{-tL}w, \qquad t\ge0.
\end{equation}

To find this exponential operator $e^{-tL}$, we interpret sequences in $\ell^2(P_\infty)$
as Fourier coefficients and transform the problem into a problem in $L^2(-\pi,\pi)$.
Define the unitary operator $\cF:\ell^2(P_\infty)\to L^2(-\pi,\pi)$ by
\[
  (\cF u)(q) = \frac{1}{\sqrt{2\pi}\,}\sum_{n\in\ZZ}u_n e^{inq},
  \qquad u=(u_n)_{n\in\ZZ}\in\ell^2(P_\infty);
\]
its inverse is given by
\[
  (\cF^{-1}g)_n = \frac{1}{\sqrt{2\pi}\,}\int_{-\pi}^\pi e^{-inq}g(q)\rd q,
  \qquad g\in L^2(-\pi,\pi).
\]
Let us first determine the operator in $L^2(-\pi,\pi)$ that is equivalent to $L_k$
via $\cF$.  For $u\in\ell^2(P_\infty)$ we have
\begin{align*}
  \bigl(\cF L_k u\bigr)(q)
  &= \frac{1}{\sqrt{2\pi}\,}\sum_{n\in\ZZ}\bigl(2u_n-u_{n-k}-u_{n+k}\bigr)e^{inq}
  \\[1ex]
  &= \frac{1}{\sqrt{2\pi}\,}\sum_{n\in\ZZ}\bigl(2u_n e^{inq}
  - u_n e^{i(n+k)q} - u_n e^{i(n-k)q}\bigr)
  \\[1ex]
  &= \frac{1}{\sqrt{2\pi}\,}\sum_{n\in\ZZ}\bigl(2-e^{ikq}-e^{-ikq}\bigr)u_n e^{inq}
  = \bigl(2-e^{ikq}-e^{-ikq}\bigr)(\cF u)(q).
\end{align*}
Hence the operator $L_k$ is unitarily equivalent to the multiplication operator
in $L^2(-\pi,\pi)$ by the function
\begin{equation}\label{defellk}
  \ell_k(q) \defeq 2-e^{ikq}-e^{-ikq} = 2\bigl(1-\cos(kq)\bigr),
\end{equation}
i.e.\
\begin{equation}\label{FLkFinv}
  \bigl(\cF L_k\cF^{-1}g\bigr)(q) = \ell_k(q)g(q),
  \qquad g\in L^2(-\pi,\pi).
\end{equation}
The transformed operators $\tLL{\lambda}$, $\tLF{z}$ and $\tLM{s}$
are also unitarily equivalent to multiplication operators:
\begin{equation}\label{FLFinv}
  \bigl(\cF\tLT\cF^{-1}g\bigr)(q) = \tell_T(q)g(q)
\end{equation}
for $T={\rm L},\lambda$, $T={\rm F},z$ or $T={\rm M},s$ where
\begin{equation}\label{deftellLFM}
  \begin{gathered}
    \tellL{\lambda}(q) \defeq \sum_{k=1}^\infty e^{-\lambda k}\ell_k(q), \hspace*{8ex}
    \tellF{z}(q) \defeq \sum_{k=1}^\infty \frac{z^k}{k!}\ell_k(q), \\
    \tellM{s}(q) \defeq \sum_{k=1}^\infty \frac{1}{k^s}\ell_k(q).
  \end{gathered}
\end{equation}
Closed forms for these sums are given in the next lemma.

\begin{lemma}
Let $\lambda\in\CC$ with $\Re\lambda>0$, $z\in\CC$, $s\in\CC$ with $\Re s>1$.
With the notation from above we have
\begin{align}
  \tellL{\lambda}(q) &= \frac{(e^\lambda+1)(1-\cos q)}{(e^\lambda-1)(\cosh\lambda-\cos q)}\,,
  \label{reptellL}
  \\[2ex]
  \tellF{z}(q) &= 2\bigl[e^z-e^{z\cos q}\cos(z\sin q)\bigr],
  \label{reptellF}
  \\[2ex]
  \tellM{s}(q) &= 2\zeta(s)-\Li_s(e^{iq})-\Li_s(e^{-iq}),
  \label{reptellM}
\end{align}
where $\Li_s$ is the polylogarithm --- also known as Jonqui\`{e}re's function --- defined
for $s\in\CC$ with $\Re s>1$ by
\[
  \Li_{s}(z) \defeq \sum_{k=1}^\infty \frac{z^k}{k^s}
  \qquad \text{when}\;\; |z|\le1
\]
and by analytic continuation to $\CC\setminus(1,\infty)$.

Moreover, the functions $\tellL{\lambda}$, $\tellF{z}$ and $\tellM{s}$ are
continuous on $[-\pi,\pi]$ and satisfy
\begin{equation}\label{lpos}
  \ell(q) > 0 \qquad \text{for}\;\;q\in[-\pi,\pi]\setminus\{0\}
\end{equation}
for $\ell=\tellL{\lambda},\tellF{z},\tellM{s}$ when $\lambda>0$, $z>0$, $s>1$, respectively.
\end{lemma}

\begin{proof}
Representation \eqref{reptellL} follows from
\begin{align*}
  \tellL{\lambda}(q)
  &= 2\sum_{k=1}^\infty e^{-\lambda k} - \sum_{k=1}^\infty e^{-\lambda k}e^{ikq}
  - \sum_{k=1}^\infty e^{-\lambda k}e^{-ikq}
  \\[1ex]
  &= \frac{2}{e^\lambda-1}-\frac{1}{e^{\lambda-iq}-1}-\frac{1}{e^{\lambda+iq}-1}
  = \frac{2}{e^\lambda-1}-2\frac{e^\lambda\cos q-1}{\bigl|e^{\lambda-iq}-1\bigr|^2}
  \\[1ex]
  &= \frac{2}{e^\lambda-1}-2\frac{e^\lambda\cos q-1}{e^{2\lambda}+1-2e^\lambda\cos q}
  = \frac{2}{e^\lambda-1} - \frac{\cos q-e^{-\lambda}}{\cosh\lambda-\cos q}
  \\[1ex]
  &= \frac{(e^\lambda+1)(1-\cos q)}{(e^\lambda-1)(\cosh\lambda-\cos q)}\,.
\end{align*}
The representations \eqref{reptellF} and \eqref{reptellM} are proved easily.
The continuity of the functions follows from the representations
\eqref{reptellL}--\eqref{reptellM} or from the uniform convergence of the series.

To show \eqref{lpos}, observe that $\ell_k(q)=2(1-\cos(kq))\ge0$ for all $q\in[-\pi,\pi]$.
Moreover, $\ell_1(q)>0$ for $q\in[-\pi,\pi]\setminus\{0\}$.  Since all coefficients
in the series in \eqref{deftellLFM} are positive when $\lambda>0$, $z>0$, $s>0$,
respectively, the claim follows.
\end{proof}

The following theorem gives an explicit description of the time evolution operator
corresponding to the transformed $k$-path Laplacians;
cf., e.g.\ \cite[Proposition~2]{CGRTV} for a similar representation for the case $L=L_1$.

\begin{theorem}\label{th:time_evol}
Let $\lambda\in\CC$ with $\Re\lambda>0$, $z\in\CC$ and $s\in\CC$ with $\Re s>1$,
let $L=L_k$, $\tLL{\lambda}$, $\tLF{z}$ or $\tLM{s}$
and let $\ell=\ell_k$, $\tellL{\lambda}$, $\tellF{z}$ or $\tellM{s}$, correspondingly.
For $w=(w_\nu)_{\nu\in\ZZ}\in\ell^2(P_\infty)$ the solution
of \eqref{time_evolutionL} is given by
\begin{equation}\label{evol_repr}
  \bigl(u(t)\bigr)_n = \bigl(e^{-tL}w\bigr)_n
  = \sum_{\nu\in\ZZ}w_\nu \frac{1}{2\pi}\int_{-\pi}^\pi e^{i(n-\nu)q}e^{-t\ell(q)}\rd q,
  \qquad t\ge0,\; n\in\ZZ.
\end{equation}
The entries of the double-infinite Toeplitz matrix corresponding to the
time evolution operator $e^{-tL}$ are
\begin{equation}\label{evol_repr2}
  \bigl(e^{-tL}\bigr)_{\mu\nu}
  = \frac{1}{2\pi}\int_{-\pi}^\pi e^{i(\mu-\nu)q}e^{-t\ell(q)}\rd q
  = \frac{1}{2\pi}\int_{-\pi}^\pi \cos\bigl((\mu-\nu)q\bigr)e^{-t\ell(q)}\rd q.
\end{equation}
\end{theorem}

\begin{proof}
Since $\cF L\cF^{-1}$ acts as a multiplication operator by $\ell$
(see \eqref{FLkFinv} and \eqref{FLFinv}), we have
\[
  \big(\cF e^{-tL}\cF^{-1}g\bigr)(q) = e^{-t\ell(q)}g(q),
  \qquad g\in L^2(-\pi,\pi).
\]
Let $\nu\in\ZZ$ and $e_\nu$ as in \eqref{defev}.  Then
\begin{align*}
  \bigl(e^{-tL}e_\nu\bigr)_n
  &= \bigl(\cF^{-1}e^{-t\ell(\cdot)}\cF e_\nu\bigr)_n
  = \frac{1}{\sqrt{2\pi}\,}\int_{-\pi}^\pi e^{-inq}e^{-t\ell(q)}
  \frac{1}{\sqrt{2\pi}\,}e^{i\nu q}\rd q
  \\[1ex]
  &= \frac{1}{2\pi}\int_{-\pi}^\pi e^{-i(n-\nu)q}e^{-t\ell(q)}\rd q
  = \frac{1}{2\pi}\int_{-\pi}^\pi e^{i(n-\nu)q}e^{-t\ell(q)}\rd q
\end{align*}
where the last equality follows since $\ell$ is an even function.
Since $e^{-tL}$ is a bounded operator we have
\[
  e^{-tL}w = \sum_{\nu\in\ZZ} w_\nu e^{-tL}e_\nu,
\]
which proves \eqref{evol_repr} and hence also \eqref{evol_repr2}.
\end{proof}

In Figure \ref{densities vs x} we illustrate the time evolution of
the density $u(t)$ for the three transforms of the $k$-path Laplace operators.
\begin{figure}
\centering
\includegraphics[width=0.48\textwidth]{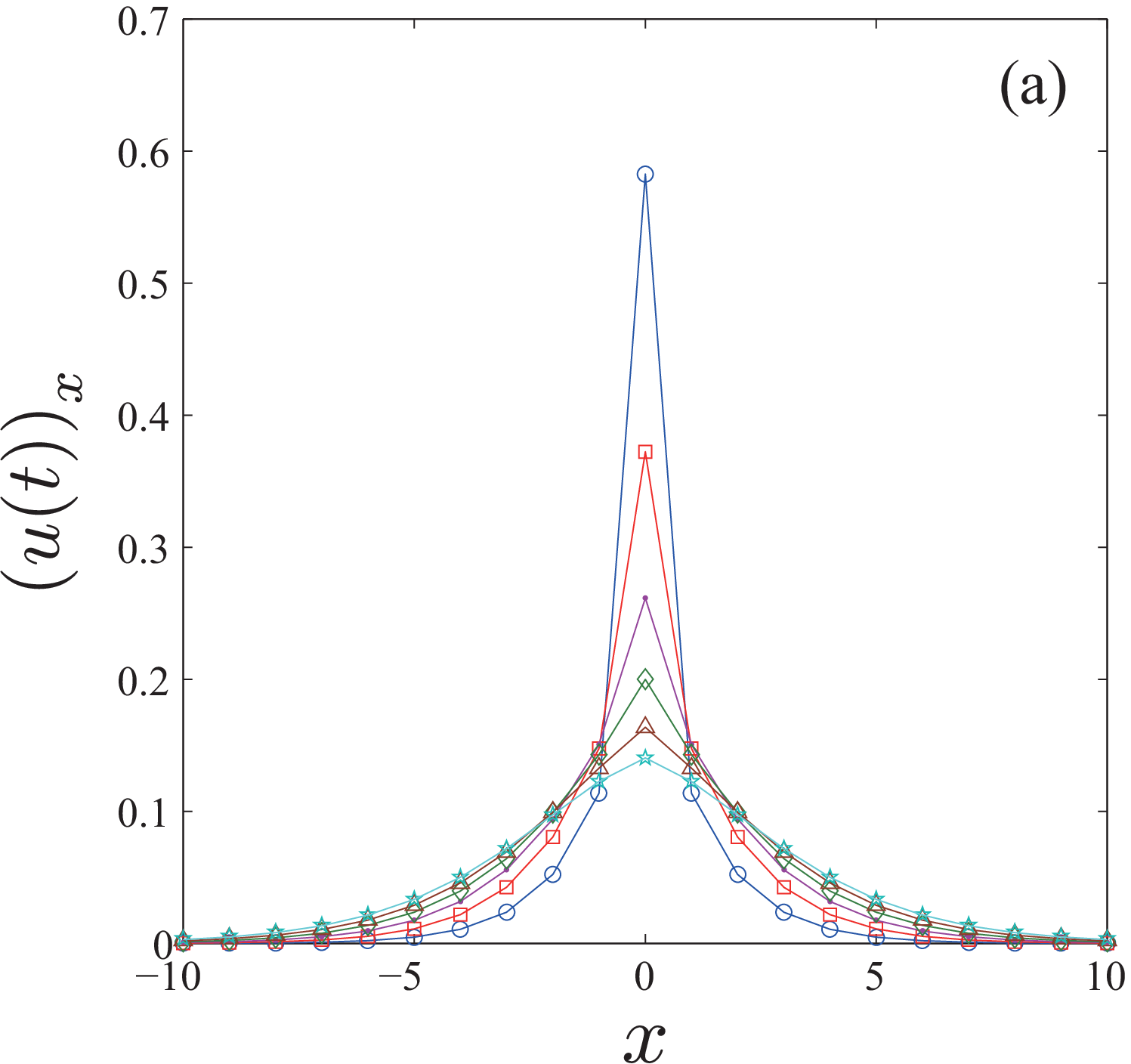}
\hfill
\includegraphics[width=0.48\textwidth]{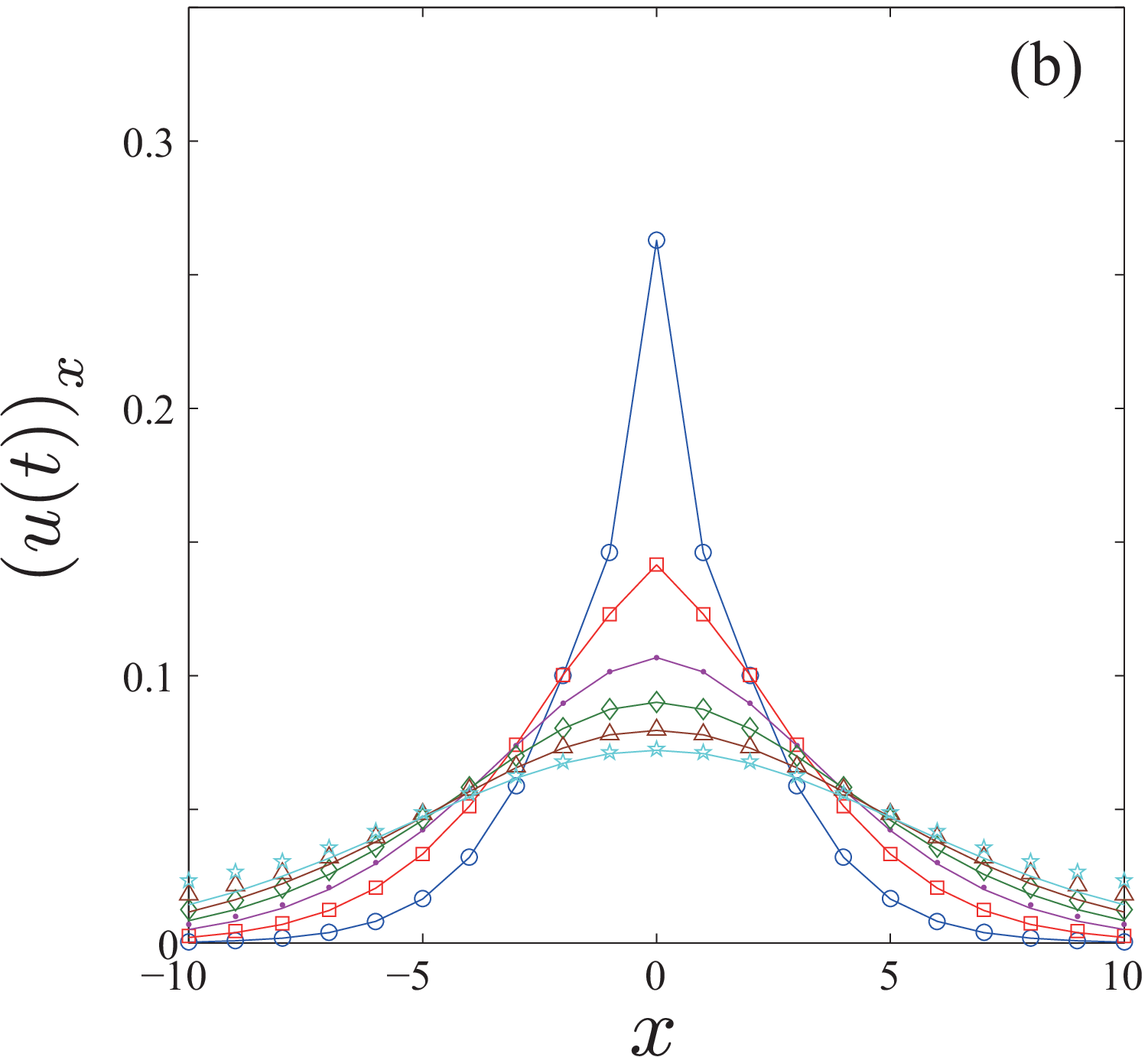}
\\
\vspace{\baselineskip}
\includegraphics[width=0.48\textwidth]{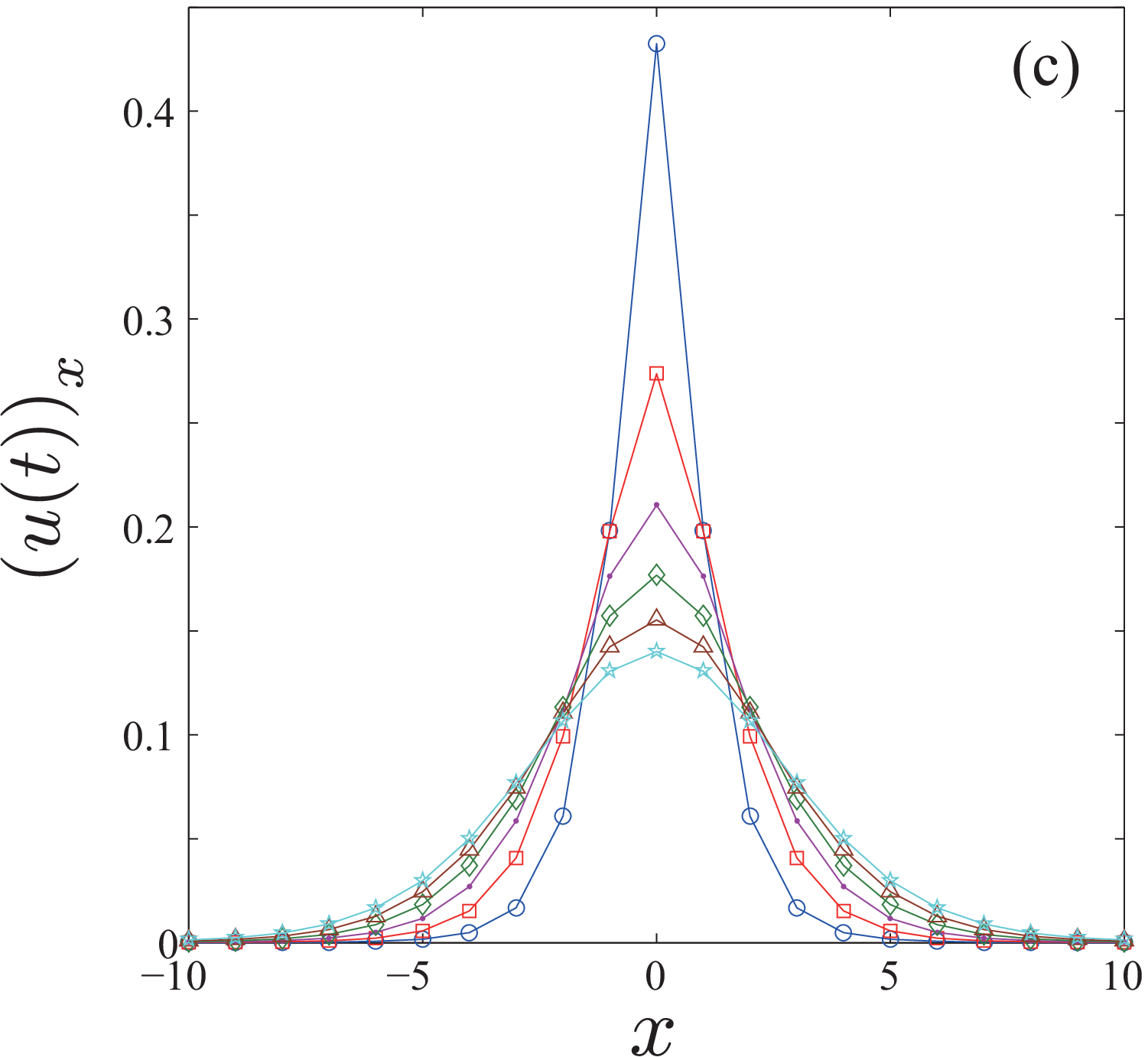}
\caption{The particle density $(u(t))_{x}$ as a function of $x$ for
the transformed $k$-path Laplacians for a linear chain with 21 nodes.
The symbols indicate the results obtained from the simulations of the
linear chain and the solid lines represent the values obtained analytically
from Theorem~\ref{th:time_evol} for different time: $t=0.5$ (circles),
$t=1.0$ (squares), $t=1.5$ (dots), $t=2.0$ (rhombus), $t=2.5$
(triangles), $t=3.0$ (stars). (a) Laplace transform with $\lambda=1$.
(b) Factorial-transform with $z=1$. (c) Mellin-transform with $s=4$.}
\label{densities vs x}
\end{figure}

\section{Generalized diffusion on the path graph}

\noindent
In this section we prove that the density profile $u(t)$ that solves
\begin{align}
  & \frac{\rd}{\rd t}u(t) = -Lu(t),
  \label{evol_equ}\\[1ex]
  & u(0) = e_0,
  \label{evol_equ_ic}
\end{align}
where $e_0$ is as in \eqref{eq:initial condition} and $L$ is any of the
transformed $k$-path Laplacians $\tLL{\lambda}$, $\tLF{z}$ or $\tLM{s}$,
approaches a stable distribution if appropriately scaled.
Stable distributions can be parameterized with four parameters as follows
(see, e.g.\ \cite[\S1.3]{nolan}).
Let $\alpha\in(0,2]$, $\beta\in[-1,1]$, $\gamma>0$ and $\delta\in\RR$;
then the density of the stable distribution $S(\alpha,\beta,\gamma,\delta)$
is given by
\[
  f(\xi;\alpha,\beta,\gamma,\delta)
  = \frac{1}{2\pi}\int_{-\infty}^\infty \phi(z;\alpha,\beta,\gamma,\delta)e^{i\xi z}\rd z,
\]
where
\[
  \phi(z;\alpha,\beta,\gamma,\delta) = \exp\Bigl[-|\gamma z|^\alpha
  \Bigl(1+i\beta\sign(z)\omega(z,\alpha)\Bigr)+i\delta z\Bigr],
  \qquad z\in\RR,
\]
with
\[
  \omega(z,\alpha) = \begin{cases}
    -\tan\dfrac{\pi\alpha}{2} & \text{if}\;\; \alpha\ne1, \\[2ex]
    \dfrac{2}{\pi}\ln|z| & \text{if}\;\; \alpha=1.
  \end{cases}
\]
Note that $\phi(\,\cdot\,;\alpha,\beta,\gamma,\delta)$ is the characteristic function
of the probability distribution $S(\alpha,\beta,\gamma,\delta)$.
We are only interested in the case when $\beta=\delta=0$, which yields the following
simpler function:
\begin{equation}\label{stable00}
  \phi(z;\alpha,0,\gamma,0) = \exp\bigl(-|\gamma z|^\alpha\bigr).
\end{equation}
There are two special cases where the density $f$ can be computed explicitly:
when $\alpha=2$, we obtain a normal distribution with variance $2\gamma^2$, i.e.\
\[
  f(\xi;2,0,\gamma,0) = \frac{1}{2\sqrt{\pi}\,\gamma}\exp\Bigl(-\frac{\xi^2}{4\gamma^2}\Bigr);
\]
when $\alpha=1$, we obtain a Cauchy distribution:
\[
  f(\xi;1,0,\gamma,0) = \frac{1}{\pi}\cdot\frac{\gamma}{\gamma^2+\xi^2}\,.
\]
When $\alpha<2$ the density has a power-like decay:
\begin{equation}\label{stable_decay}
  f(\xi;\alpha,0,\gamma,0) \sim \frac{1}{\pi}\Gamma(\alpha+1)
  \sin\Bigl(\frac{\pi\alpha}{2}\Bigr)\gamma^\alpha\cdot\frac{1}{\xi^{\alpha+1}}
  \qquad\text{as}\;\; \xi\to\pm\infty;
\end{equation}
see, e.g.\ \cite[Theorem~1.12]{nolan}.
Here and in the following we use the following notation:
let $g_1$ and $g_2$ be functions that are defined and positive-valued on
an interval of the form $(a,\infty)$; we write
\[
  g_1(x) \sim g_2(x) \qquad \text{as}\;\; x\to\infty
  \hspace*{8ex} \text{if} \qquad
  \lim_{x\to\infty}\frac{g_1(x)}{g_2(x)} = 1;
\]
a similar notation is used for the behaviour as $x\to0$.
Relation \eqref{stable_decay} implies that $S(\alpha,0,\gamma,0)$
has no finite variance if $\alpha<2$.
Further, if $\alpha\le1$, then even the first moment is infinite.

In the following we study the asymptotic behaviour of the density profile $u(t)$
that solves \eqref{evol_equ} and \eqref{evol_equ_ic}.
It follows from \eqref{evol_repr} that
\begin{equation}\label{evol_e0}
  \bigl(u(t)\bigr)_x = \bigl(e^{-tL}e_0\bigr)_x
  = \frac{1}{2\pi}\int_{-\pi}^\pi e^{ixq}e^{-t\ell(q)}\rd q,
  \qquad x\in\ZZ,\, t\ge0.
\end{equation}

In the next lemma we consider the asymptotic behaviour of integrals as in \eqref{evol_e0}.
We allow also values $\alpha>2$ in the notation $f(z;\alpha,0,\gamma,0)$
although this is not needed later.

\begin{lemma}\label{le:asymp}
Let $\alpha>0$ and let $h:[-\pi,\pi]\to\RR$ be a continuous function that satisfies
\begin{alignat}{2}
  & h(q) > 0 \qquad && \text{for}\;\; q\in[-\pi,\pi]\setminus\{0\},
  \label{lemasymp_assump1}\\[1ex]
  & h(q) \sim c|q|^\alpha \qquad && \text{as}\;\; q\to0
  \label{lemasymp_assump2}
\end{alignat}
with some $c>0$.  Then
\begin{align}
  t^{\frac{1}{\alpha}}\frac{1}{2\pi}\int_{-\pi}^\pi e^{it^{\frac{1}{\alpha}}\xi q}
  e^{-t h(q)}\rd q
  &\to \frac{1}{2\pi}\int_{-\infty}^\infty e^{i\xi z}e^{-c|z|^\alpha}\rd z
  \label{conv_stable}\\[1ex]
  &= f\bigl(\xi;\alpha,0,c^{\frac{1}{\alpha}},0\bigr)
  \label{conv_stable_f}
\end{align}
uniformly in $\xi$ on $\RR$ as $t\to\infty$.

Hence
\begin{equation}\label{lem_sim}
  \frac{1}{2\pi}\int_{-\pi}^\pi e^{ixq}e^{-th(q)}\rd q
  = t^{-\frac{1}{\alpha}}f\bigl(t^{-\frac{1}{\alpha}}x;\alpha,0,c^{\frac{1}{\alpha}},0\bigr)
  + \rmo\bigl(t^{-\frac{1}{\alpha}}\bigr)
\end{equation}
uniformly in $x\in\RR$ as $t\to\infty$.
\end{lemma}

\begin{proof}
Let $t>0$.
With the substitution $z=t^{\frac{1}{\alpha}}q$ we have
\begin{align}
  & \Biggl|\,t^{\frac{1}{\alpha}}\frac{1}{2\pi}\int_{-\pi}^\pi e^{it^{\frac{1}{\alpha}}\xi q}
  e^{-t h(q)}\rd q
  - \frac{1}{2\pi}\int_{-\infty}^\infty e^{i\xi z}e^{-c|z|^\alpha}\rd z\Biggr|
  \notag\\[1ex]
  &= \Biggl|\frac{1}{2\pi}\int_{-\pi t^{\frac{1}{\alpha}}}^{\pi t^{\frac{1}{\alpha}}}
  e^{i\xi z}e^{-th(t^{-\frac{1}{\alpha}}z)}\rd z
  - \frac{1}{2\pi}\int_{-\infty}^\infty e^{i\xi z}e^{-c|z|^\alpha}\rd z\Biggr|
  \notag\\[1ex]
  &\le \Biggl|\frac{1}{2\pi}\int_{-\pi t^{\frac{1}{\alpha}}}^{\pi t^{\frac{1}{\alpha}}}
  e^{i\xi z}\Bigl(e^{-th(t^{-\frac{1}{\alpha}}z)}-e^{-c|z|^\alpha}\Bigr)\rd z\Biggr|
  \notag\\[1ex]
  &\quad + \Biggl|\frac{1}{2\pi}
  \int\limits_{\RR\setminus[-\pi t^{\frac{1}{\alpha}},\pi t^{\frac{1}{\alpha}}]}
  e^{i\xi z}e^{-c|z|^\alpha}\rd z\Biggr|
  \notag\displaybreak[0]\\[1ex]
  &\le \frac{1}{2\pi}\int_{-\pi t^{\frac{1}{\alpha}}}^{\pi t^{\frac{1}{\alpha}}}
  \Bigl|e^{-t h(t^{-\frac{1}{\alpha}}z)}-e^{-c|z|^\alpha}\Bigr|\rd z
  \label{int1}\\[1ex]
  &\quad + \frac{1}{2\pi}
  \int\limits_{\RR\setminus[-\pi t^{\frac{1}{\alpha}},\pi t^{\frac{1}{\alpha}}]}
  e^{-c|z|^\alpha}\rd z.
  \label{int2}
\end{align}
First note that the integrals in \eqref{int1} and \eqref{int2} are independent of $\xi$.
We show that both integrals converge to $0$ as $t\to\infty$.
For the integral in \eqref{int2} this is clear.
Let us now consider the integral in \eqref{int1}.
Since $h$ is continuous and satisfies \eqref{lemasymp_assump1} and \eqref{lemasymp_assump2},
the function $q\mapsto h(q)/|q|^\alpha$ is bounded below by a positive constant,
i.e.\ there exists $\tilde c>0$ such that
\[
  h(q) \ge \tilde c|q|^\alpha \qquad\text{for}\;\; q\in[-\pi,\pi].
\]
This implies that the integrand in \eqref{int1} satisfies
\begin{align*}
  & \Bigl|e^{-t h(t^{-\frac{1}{\alpha}}z)}-e^{-c|z|^\alpha}\Bigr|
  \le e^{-t h(-t^{-\frac{1}{\alpha}}z)} + e^{-c|z|^\alpha}
  \\[1ex]
  &\le e^{-t\tilde c|t^{-\frac{1}{\alpha}}z|^\alpha} + e^{-c|z|^\alpha}
  = e^{-\tilde c|z|^\alpha} + e^{-c|z|^\alpha}
\end{align*}
for $z\in[-\pi t^{\frac{1}{\alpha}},\pi t^{\frac{1}{\alpha}}]$.
Therefore the integrand in \eqref{int1} is bounded by the integrable function
$z\mapsto e^{-\tilde c|z|^\alpha} + e^{-c|z|^\alpha}$, which is independent of $t$.
For fixed $z\in\RR$ we have
\[
  t h\bigl(t^{-\frac{1}{\alpha}}z\bigr)
  = |z|^\alpha\frac{h(t^{-\frac{1}{\alpha}}z)}{|t^{-\frac{1}{\alpha}}z|^\alpha}
  \to c|z|^\alpha \qquad \text{as}\;\; t\to\infty
\]
by \eqref{lemasymp_assump2} and hence
\[
  \Bigl|e^{-t h(t^{-\frac{1}{\alpha}}z)}-e^{-c|z|^\alpha}\Bigr|
  \to 0 \qquad \text{as}\;\; t\to\infty.
\]
Now the Dominated Convergence Theorem implies that the integral in \eqref{int1}
converges to $0$ as $t\to\infty$.  This shows \eqref{conv_stable}.

Finally, we prove \eqref{lem_sim}.
With the substitution $x=t^{\frac{1}{\alpha}}\xi$
we obtain from \eqref{conv_stable} and \eqref{conv_stable_f} that
\begin{align*}
  & \frac{1}{t^{-\frac{1}{\alpha}}}
  \Biggl|\frac{1}{2\pi}\int_{-\pi}^\pi e^{ixq}e^{-th(q)}\rd q
  - t^{-\frac{1}{\alpha}}f\bigl(t^{-\frac{1}{\alpha}}x;\alpha,0,c^{\frac{1}{\alpha}},0\bigr)\Biggr|
  \\[1ex]
  &= \Biggl|t^{\frac{1}{\alpha}}\frac{1}{2\pi}\int_{-\pi}^\pi e^{ixq}e^{-th(q)}\rd q
  - f\bigl(t^{-\frac{1}{\alpha}}x;\alpha,0,c^{\frac{1}{\alpha}},0\bigr)\Biggr|
  \to 0
\end{align*}
uniformly in $x\in\RR$ as $t\to\infty$, which shows \eqref{lem_sim}.
\end{proof}

\begin{remark}\label{re:scaling}
The lemma can be interpreted as follows.  If the function
\[
  g(x,t) \defeq \frac{1}{2\pi}\int_{-\pi}^\pi e^{ixq}e^{-th(q)}\rd q
\]
is scaled in the independent and the dependent variable, then it converges:
\[
  t^{\frac{1}{\alpha}}g\bigl(t^{\frac{1}{\alpha}}\xi,t\bigr)
  \to f\bigl(\xi;\alpha,0,c^{\frac{1}{\alpha}},0\bigr)
  \qquad\text{as}\;\; t\to\infty.
\]
This means that the profile spreads proportionally to $t^{\frac{1}{\alpha}}$.
The solution $(u(t))_x$ is defined only for $x\in\ZZ$.
Scaling this discrete profile in the same way leads to a sequence of points:
\[
  \Bigl(t^{-\frac{1}{\alpha}}x,t^{\frac{1}{\alpha}}\bigl(u(t)\bigr)_x\Bigr),
  \qquad x\in\ZZ,
\]
for each $t\ge0$; the points lie on the graph of the
function $\xi\mapsto t^{\frac{1}{\alpha}}g(t^{\frac{1}{\alpha}}\xi,t)$.
These sequences of points become denser as $t$ growths and converge
to the limiting profile $f(\xi;\alpha,0,c^{\frac{1}{\alpha}},0)$ as $t\to\infty$.
In particular, the maximum height, which is attained at $0$, decreases like
\begin{equation}\label{peak}
  \bigl(u(t)\bigr)_0 \sim t^{-\frac{1}{\alpha}}f\bigl(0;\alpha,0,c^{\frac{1}{\alpha}},0\bigr)
  = \frac{\Gamma\bigl(\frac{\alpha+1}{\alpha}\bigr)}{\pi c^{\frac{1}{\alpha}}}
  t^{-\frac{1}{\alpha}},
  \qquad\text{as}\;\; t\to\infty.
\end{equation}
The full width at half maximum (FWHM) increases like
\begin{equation}\label{FWHM}
  \FWHM(t) \sim 2\xi_0t^{\frac{1}{\alpha}}
  \qquad\text{as}\;\; t\to\infty,
\end{equation}
where $\xi_0>0$ is such that
$f(\xi_0;\alpha,0,c^{\frac{1}{\alpha}},0)=\frac{1}{2}f(0;\alpha,0,c^{\frac{1}{\alpha}},0)$.
This implies that if $\alpha=2$, then one has normal diffusion,
and if $\alpha<2$, then the time evolution is superdiffusive
since $(\FWHM(t))^2\sim c t^\kappa$ with $\kappa=\frac{2}{\alpha}$.
We used the square of the full width at half maximum $\FWHM^2$ instead of
the mean square displacement $\mathrm{MSD}$ because the latter is infinite
if $\alpha<2$.
\end{remark}

\subsection{Diffusion for the Laplace- and factorial-transformed {\boldmath$k$}-path Laplacians}

In the next theorem we show that the time evolution with the $k$-path Laplacians
and the Laplace-transformed and factorial-transformed $k$-path Laplacians show
normal diffusion (see, e.g.\ \cite{delmotte99} for the case $L=L_1$).
This is caused by the fact that $\ell_k$, $\tellL{\lambda}$
and $\tellF{z}$ behave quadratically around $0$.

\begin{theorem}\label{th:asympLaplFact}
Let $P_\infty$ be the infinite path graph, let $\lambda,z>0$ and
let $\tLL{\lambda}$ and $\tLF{z}$ be the Laplace-transformed and
factorial-transformed $k$-path Laplacian with parameters $\lambda$ and $z$, respectively.
Moreover, let $u(t)$ be the solution of \eqref{evol_equ}, \eqref{evol_equ_ic}
with $L=L_k$, $L=\tLL{\lambda}$ or $L=\tLF{z}$.  Then
\begin{equation}\label{asymp_LF}
  \bigl(u(t)\bigr)_x = t^{-\frac{1}{2}}\frac{1}{2\sqrt{\pi a}\,}
  \exp\biggl(-\frac{x^2}{4at}\biggr) + \rmo\bigl(t^{-\frac{1}{2}}\bigr)
  \qquad \text{as}\;\; t\to\infty
\end{equation}
uniformly in $x\in\ZZ$ where
\begin{alignat}{2}
  a &= k^2 \qquad & & \text{for}\;\; L=L_k,
  \label{ak}\\[1ex]
  a &= \frac{e^{\lambda}\left(e^{\lambda}+1\right)}{(e^{\lambda}-1)^{3}}
  = \frac{\coth\bigl(\frac{\lambda}{2}\bigr)}{2(\cosh\lambda-1)}
  \qquad &
  & \text{for}\;\; L=\tLL{\lambda},
  \label{aL}\\[1ex]
  a &= z(z+1)e^{z} \qquad &
  & \text{for}\;\; L=\tLF{z}.
  \label{aF}
\end{alignat}
\end{theorem}

\begin{proof}
The asymptotic behaviour of the functions $\ell_k$, $\tellL{\lambda}$ and $\tellF{z}$
from \eqref{defellk}, \eqref{reptellL} and \eqref{reptellF} is $\ell_k(q)\sim k^2q^2$
as $q\to0$,
\[
  \tellL{\lambda}(q)
  = \frac{e^\lambda+1}{(e^\lambda-1)(\cosh\lambda-1)}\cdot\frac{q^2}{2} + \rmO\bigl(q^4\bigr)
  = aq^2 + \rmO\bigl(q^4\bigr) \qquad \text{as}\;\; q\to0
\]
with $a$ from \eqref{aL} and
\begin{align*}
  \tellF{z}(q) &= 2\biggl[e^z-e^z\biggl(1-\frac{zq^2}{2}+\rmO\bigl(q^4\bigr)\biggr)
  \biggl(1-\frac{z^2q^2}{2}+\rmO\bigl(q^4\bigr)\biggr)\biggr]
  \\[1ex]
  &= e^z(z+z^2)q^2 + \rmO\bigl(q^4\bigr)
  \qquad \text{as}\;\; q\to0
\end{align*}
with $a$ from \eqref{aF}, respectively.
Now \eqref{asymp_LF} follows from \eqref{evol_e0} and Lemma~\ref{le:asymp}.
\end{proof}

\begin{figure}
\centering
\includegraphics[width=0.6\textwidth]{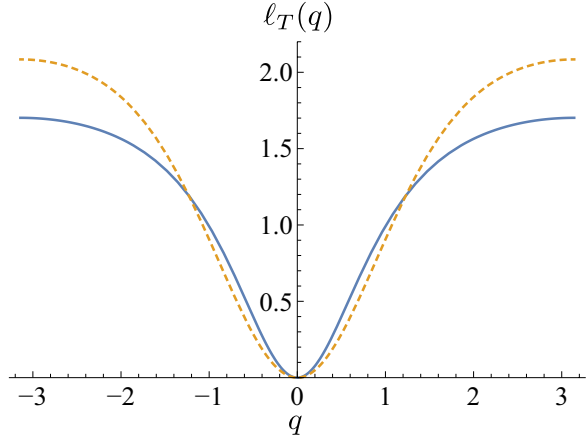}
\caption{The graphs of the functions $\tellL{\lambda}$ with $\lambda=1$ (blue solid curve)
and $\tellF{z}$ with $z=1/2$ (orange broken curve) on the interval $[-\pi,\pi]$.}
\label{fig3}
\end{figure}

\begin{remark}
Theorem~\ref{th:asympLaplFact} shows that the diffusion for the $k$-path Laplacian,
the Laplace-transformed and the factorial-transformed $k$-path Laplacian
is always normal.
The peak height of the distribution is attained at $x=0$ and behaves like
\[
  \bigl(u(t)\bigr)_0 \sim \frac{1}{2\sqrt{\pi a}\,}t^{-\frac{1}{2}}
  \qquad\text{as}\;\; t\to\infty,
\]
where $a$ is from \eqref{ak}--\eqref{aF}; see \eqref{peak}.
The mean square displacement behaves like
\[
  \mathrm{MSD}(t) \sim 2at \qquad\text{as}\;\; t\to\infty
\]
and the full width at half maximum (FWHM) behaves like
\[
  \FWHM(t) \sim 2\sqrt{(\ln 2)a}\,t^{\frac{1}{2}}
  \qquad\text{as}\;\; t\to\infty;
\]
see \eqref{FWHM}.
For the limiting behaviour after rescaling in $x$ see Remark~\ref{re:scaling}.
\end{remark}

\begin{figure}
\centering
\includegraphics[width=0.48\textwidth]{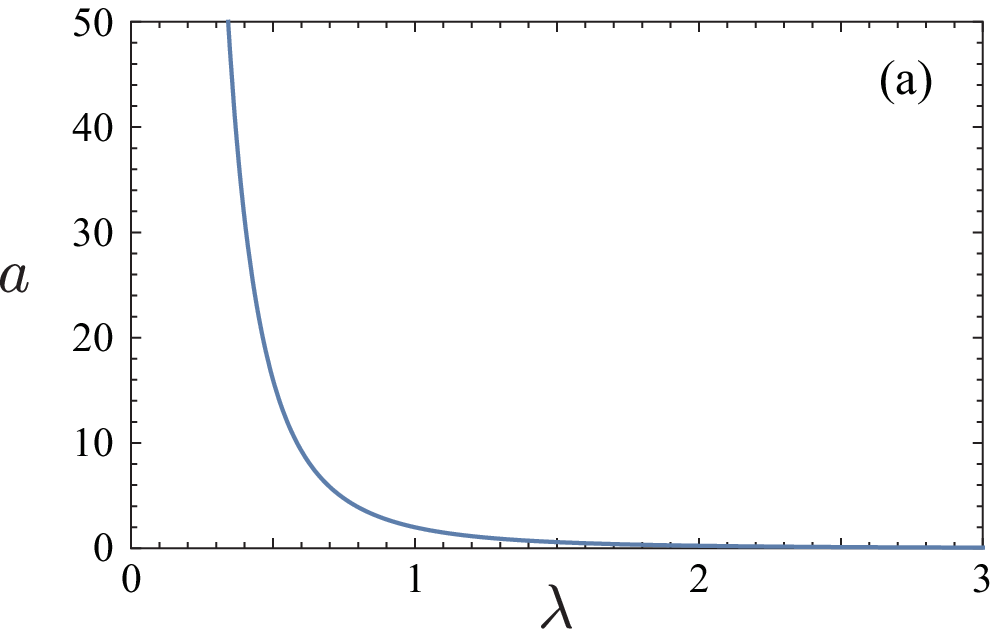}
\hfill
\includegraphics[width=0.48\textwidth]{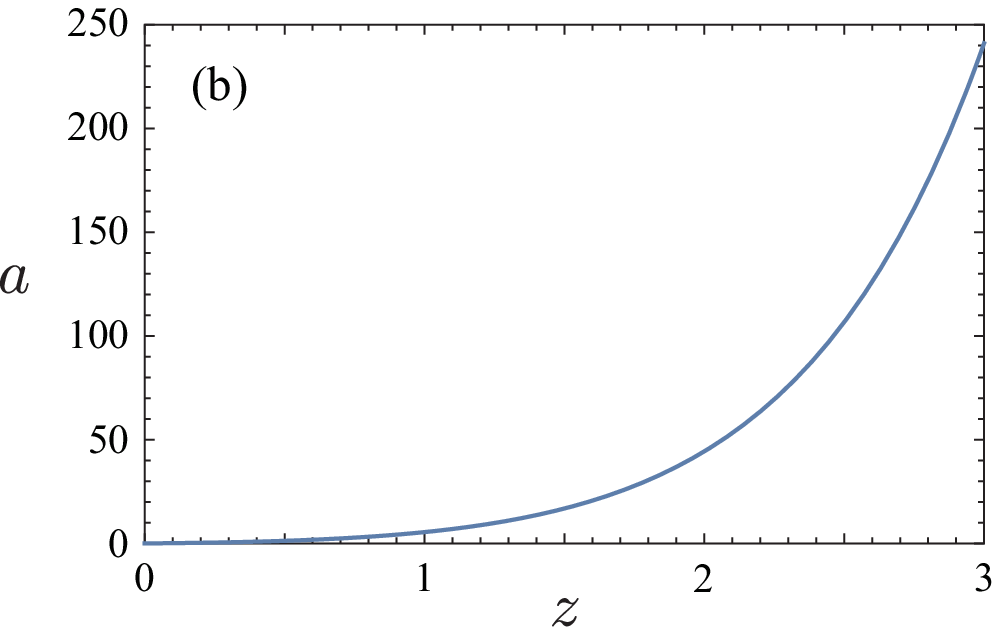}
\caption{The parameter dependence of $a$ for (a) the Laplace-transformed and
(b) the factorial-transformed $k$-path Laplacian.}
\label{fig4}
\end{figure}

\subsection{Diffusion for the Mellin-transformed {\boldmath$k$}-path Laplacian}

For the Mellin-transformed $k$-path Laplacian, the density profile shows
superdiffusion for $1<s<3$ and normal diffusion for $s>3$.

\begin{theorem}\label{th:asympMellin}
Let $P_\infty$ be the infinite path graph, let $s>1$ and let $\tLM{s}$ be the
Mellin-transformed $k$-path Laplacian with parameter $s$.
Moreover, let $u(t)$ be the solution of \eqref{evol_equ} and \eqref{evol_equ_ic}
with $L=\tLM{s}$.  Then
\begin{equation}\label{asymp_mellin}
  \bigl(u(t)\bigr)_x = t^{-\frac{1}{\alpha}}f\bigl(t^{-\frac{1}{\alpha}}x;\alpha,0,\gamma,0\bigr)
  + \rmo\bigl(t^{-\frac{1}{\alpha}}\bigr)
  \qquad \text{as}\;\; t\to\infty
\end{equation}
uniformly in $x\in\ZZ$ where
\begin{alignat}{3}
  \alpha &= s-1, \quad &
  \gamma &= \biggl(-\frac{\pi}{\Gamma(s)\cos\bigl(\frac{\pi s}{2}\bigr)}\biggr)^{\frac{1}{s-1}}
  \qquad &
  & \text{if}\;\; 1<s<3,
  \label{alphagamma1}\\[1ex]
  \alpha &= 2, \quad &
  \gamma &= \sqrt{\zeta(s-2)} \qquad &
  & \text{if}\;\; s>3.
  \label{alphagamma2}
\end{alignat}
In the case $1<s<3$, the (rescaled) limit distribution has the following
asymptotic behaviour:
\begin{equation}\label{asymp_xi_large_mellin}
  f(\xi;\alpha,0,\gamma,0) \sim \frac{1}{\xi^s}
  \qquad\text{as}\;\; \xi\to\pm\infty,
\end{equation}
where $\alpha$ and $\gamma$ are as in \eqref{alphagamma1}.
\end{theorem}

Note that in the case when $s>3$ the limiting distribution is a normal distribution
and hence
\[
  \bigl(u(t)\bigr)_x = t^{-\frac{1}{2}}\frac{1}{2\sqrt{\pi\zeta(s-2)}\,}
  \exp\biggl(-\frac{x^2}{4\zeta(s-2)t}\biggr) + \rmo\bigl(t^{-\frac{1}{2}}\bigr)
  \qquad \text{as}\;\; t\to\infty;
\]
when $s=2$, the limiting distribution is a Cauchy distribution and hence
\[
  \bigl(u(t)\bigr)_x = \frac{t}{x^2+\pi^2t^2} + \rmo\bigl(t^{-1}\bigr)
  \qquad\text{as}\;\; t\to\infty.
\]

\begin{proof}
We consider the behaviour of $\tellM{s}$ from \eqref{reptellM} at $0$.
Let $s>1$ with $s\notin\NN$.
It follows from \cite[25.12.12]{nist_handbook} that
\[
  \Li_s(e^z) = \Gamma(1-s)(-z)^{s-1} + \sum_{n=0}^\infty \zeta(s-n)\frac{z^n}{n!}\,,
  \qquad |z|<2\pi,\; z\notin(0,\infty),
\]
which yields
\begin{align}
  \tellM{s}(q) &= 2\zeta(s)-\Li_s(e^{iq})-\Li_s(e^{-iq})
  \notag\\[1ex]
  &= 2\zeta(s)-\Gamma(1-s)\Bigl((-iq)^{s-1}+(iq)^{s-1}\Bigr)
  - \sum_{n=0}^\infty \zeta(s-n)\frac{(iq)^n+(-iq)^n}{n!}
  \displaybreak[0]\notag\\[1ex]
  &= -\Gamma(1-s)|q|^{s-1}\Bigl(e^{-i(s-1)\frac{\pi}{2}}+e^{i(s-1)\frac{\pi}{2}}\Bigr)
  - \sum_{n=1}^\infty \zeta(s-n)\frac{(iq)^n+(-iq)^n}{n!}
  \notag\displaybreak[0]\\[1ex]
  &= -\frac{\pi}{\Gamma(s)\sin(\pi s)}|q|^{s-1}2\cos\biggl(\frac{(s-1)\pi}{2}\biggr)
  - \sum_{l=1}^\infty \zeta(s-2l)\frac{2(-1)^l q^{2l}}{(2l)!}
  \notag\\[1ex]
  &= -\frac{\pi}{\Gamma(s)\cos\bigl(\frac{s\pi}{2}\bigr)}|q|^{s-1}
  - \sum_{l=1}^\infty \zeta(s-2l)\frac{2(-1)^l q^{2l}}{(2l)!}
  \label{expan_ellM1}\\[1ex]
  &= -\frac{\pi}{\Gamma(s)\cos\bigl(\frac{s\pi}{2}\bigr)}|q|^{s-1}
  + \zeta(s-2)q^2 + \rmO(q^4)
  \label{expan_ellM2}
\end{align}
as $q\to0$.  By continuity \eqref{expan_ellM1} and hence \eqref{expan_ellM2} are
also valid for $s=2$.
If $s<3$, then the first term in \eqref{expan_ellM2} is dominating; if $s>3$, then
the second term is dominating.  Hence
\[
  \tellM{s}(q) \sim \begin{cases}
    -\dfrac{\pi}{\Gamma(s)\cos\bigl(\frac{s\pi}{2}\bigr)}|q|^{s-1}
    & \text{if}\;\; 1<s<3,
    \\[3ex]
    \zeta(s-2)q^2 & \text{if}\;\; s>3
  \end{cases}
\]
as $q\to0$.
Now \eqref{asymp_mellin} follows from \eqref{evol_e0} and Lemma~\ref{le:asymp}.

To show \eqref{asymp_xi_large_mellin}, we use \eqref{stable_decay}, which yields
\[
  f(\xi;s-1,0,\gamma,0)
  \sim \frac{1}{\pi}\Gamma(s)\sin\biggl(\frac{\pi(s-1)}{2}\biggr)
  \frac{-\pi}{\Gamma(s)\cos\bigl(\frac{\pi s}{2}\bigr)}\cdot\frac{1}{\xi^s}
  = \frac{1}{\xi^s}
\]
as $\xi\to\pm\infty$.
\end{proof}

When $s=3$, the asymptotic expansion of $\tellM{s}(q)$ involves a logarithmic term,
which implies that the asymptotic behaviour of $(u(t))_x$ is more complicated.

\begin{figure}
\centering
\includegraphics[width=0.6\textwidth]{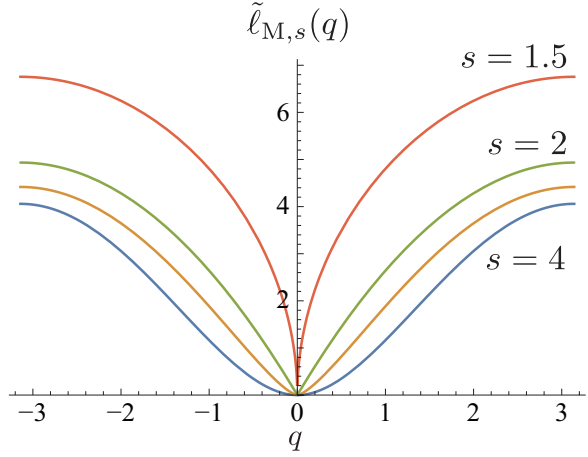}
\caption{The graphs of the functions $\tellM{s}$ on the interval $[-\pi,\pi]$.
The parameter $s$ is varied from top to the bottom
as $1.5$, $2$, $2.5$ (not labelled) and $4$.}
\label{fig6}
\end{figure}

\begin{remark}
In Fig.~\ref{fig7} we plot the density profiles for various times when the
time evolution is governed by \eqref{evol_equ} with $L$ being the Mellin-transformed
$k$-path Laplacian $\tLM{s}$.  The peak height is attained at $x=0$ and behaves like
\[
  \bigl(u(t)\bigr)_0 \sim
  \begin{cases}
    \dfrac{\Gamma\bigl(\frac{s}{s-1}\bigr)}{\pi\gamma}\,t^{-\frac{1}{s-1}}
    & \text{if}\;\; 1<s<3,
    \\[3ex]
    \dfrac{1}{2\sqrt{\pi\zeta(s-2)}\,}\,t^{-\frac{1}{2}}
    & \text{if}\;\; s>3
  \end{cases}
\]
as $t\to\infty$, where $\gamma$ is as in \eqref{alphagamma1};
see \eqref{peak}.
If $s\in(1,3)$, then the full width at half maximum (FWHM) behaves like
\[
  \FWHM(t) \sim 2\xi_0 t^{\frac{1}{s-1}}
  \qquad\text{as}\;\; t\to\infty,
\]
where $\xi_0>0$ is such that $f(\xi_0;s-1,0,\gamma,0)=\frac{1}{2}f(0;s-1,0,\gamma,0)$;
see \eqref{FWHM}.
This shows that we have superdiffusion in this case since $\frac{1}{s-1}>\frac{1}{2}$\,.
A particular case is when $s=2$, when the (rescaled) limit distribution is a
Cauchy distribution and the FWHM grows linearly, i.e.\ the time evolution shows
ballistic diffusion.
For an interpretation of the limiting behaviour using rescaling in $x$
see Remark~\ref{re:scaling}.
\end{remark}

\begin{remark}
Consider the operator
\begin{equation}\label{fract_power}
  L = cL_1^a
\end{equation}
with $c>0$ and $a\in(0,1)$, i.e.\ $L_1^a$ is a fractional power of the standard
Laplacian $L_1$ defined, e.g.\ by the spectral theorem.
Since the operator $cL_1^a$ is equivalent to the multiplication operator by
\[
  \ell(q) = c\bigl(2(1-\cos q)\bigr)^a
\]
in the Fourier representation, we obtain from Lemma~\ref{le:asymp} that
\[
  \bigl(u(t)\bigr)_x = t^{-\frac{1}{2a}}f\bigl(t^{-\frac{1}{2a}}x;2a,0,c^{\frac{1}{2a}},0\bigr)
  + \rmo\bigl(t^{-\frac{1}{2a}}\bigr) \qquad\text{as}\;\; t\to\infty
\]
when $u$ is a solution of \eqref{evol_equ}, \eqref{evol_equ_ic} with $L$
as in \eqref{fract_power}.  Hence if we choose
\[
  a = \frac{s-1}{2} \qquad\text{and}\qquad c = -\frac{\pi}{\Gamma(s)\cos\bigl(\frac{\pi s}{2}\bigr)}
\]
for $s\in(1,3)$, we obtain the same asymptotic behaviour of $u$ as the solution
in Theorem~\ref{th:asympMellin}.
However, the solutions behave differently for small $t$ as can be seen from
Figure~\ref{fig9} where the blue solid line with circles corresponds to $L=\tLM{s}$
and the red dashed line with squares corresponds to $L=cL_1^a$
for $t=1$ (a) and $t=3$ (b).
See \cite{CLRV15} for a discussion of $L_1^a$ where it was also shown
that \eqref{evol_equ}, \eqref{evol_equ_ic} with $L=L_1^a$ is equivalent
to an evolution equation with a fractional time derivative (\cite[Theorem~3]{CLRV15}).
\end{remark}

\noindent
\textbf{Acknowledgement.}
We would like to thank an anonymous referee for many valuable suggestions.
EH thanks the Ministry of Higher Education and Scientific Research
of Iraq for a PhD Scholarship.

\begin{figure}
\includegraphics[width=0.48\textwidth]{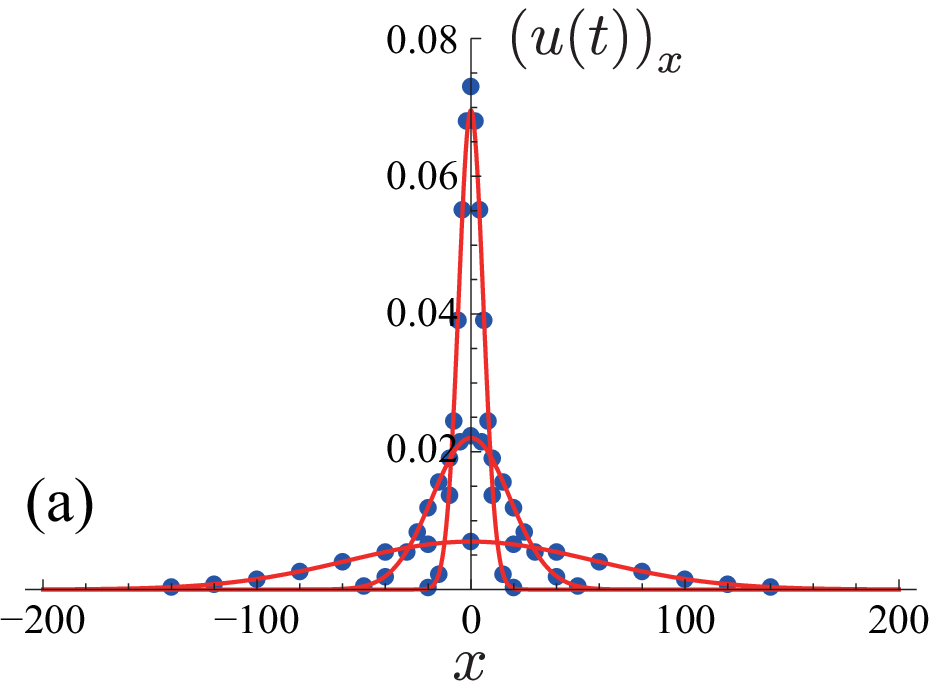}
\hfill
\includegraphics[width=0.48\textwidth]{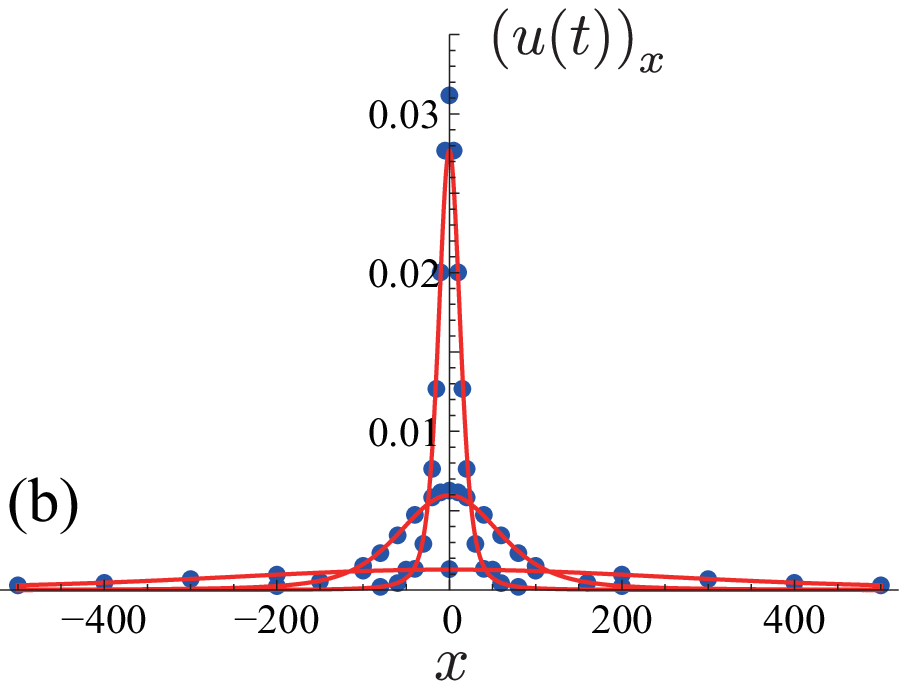}
\\
\vspace{\baselineskip}
\includegraphics[width=0.48\textwidth]{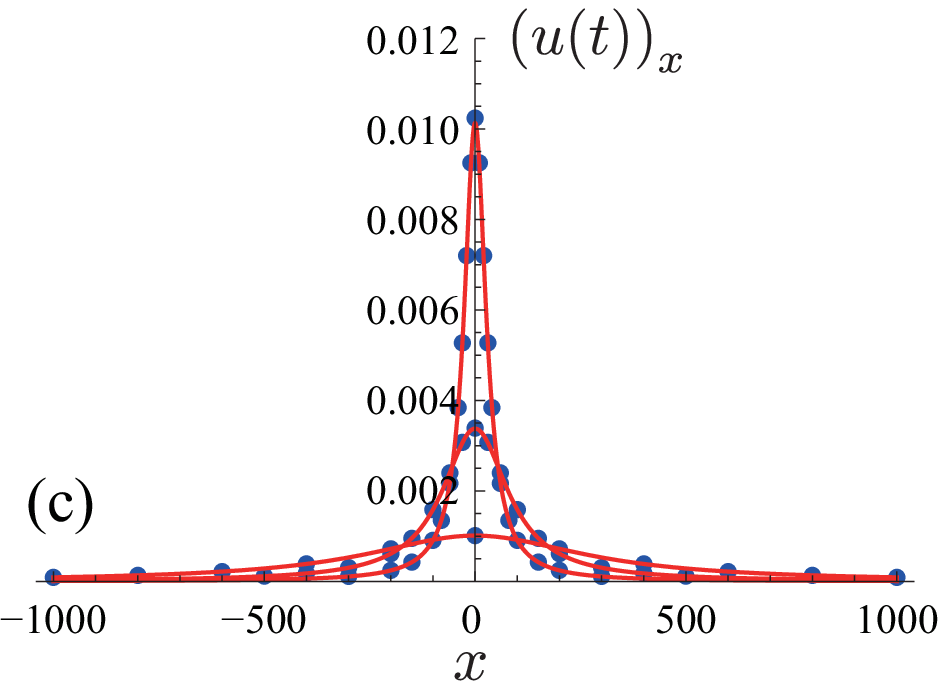}
\hfill
\includegraphics[width=0.48\textwidth]{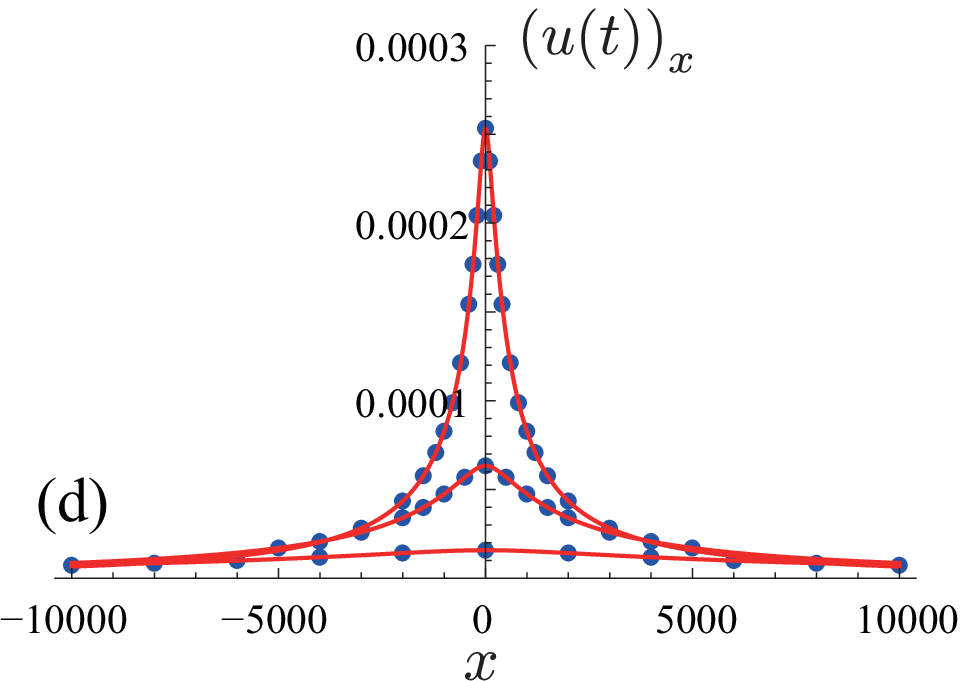}
\caption{The time evolution of the density profile under the
Mellin-transformed $k$-path Laplacian:
(a) $s=4$ for $t=10,100,1000$ from high to low;
(b) $s=2.5$ for  $t=10,100,1000$ from high to low;
(c) $s=2$ for $t=10,30,100$ from high to low;
(d) $s=1.5$ for $t=10,20,40$ from high to low.
In every panel, the blue dots indicate the result of numerical integration
of \eqref{evol_e0} with $\ell=\tellM{s}$,
whereas the red curves indicate the asymptote \eqref{asymp_mellin}.}
\label{fig7}
\end{figure}

\begin{figure}
\centering
\includegraphics[width=0.48\textwidth]{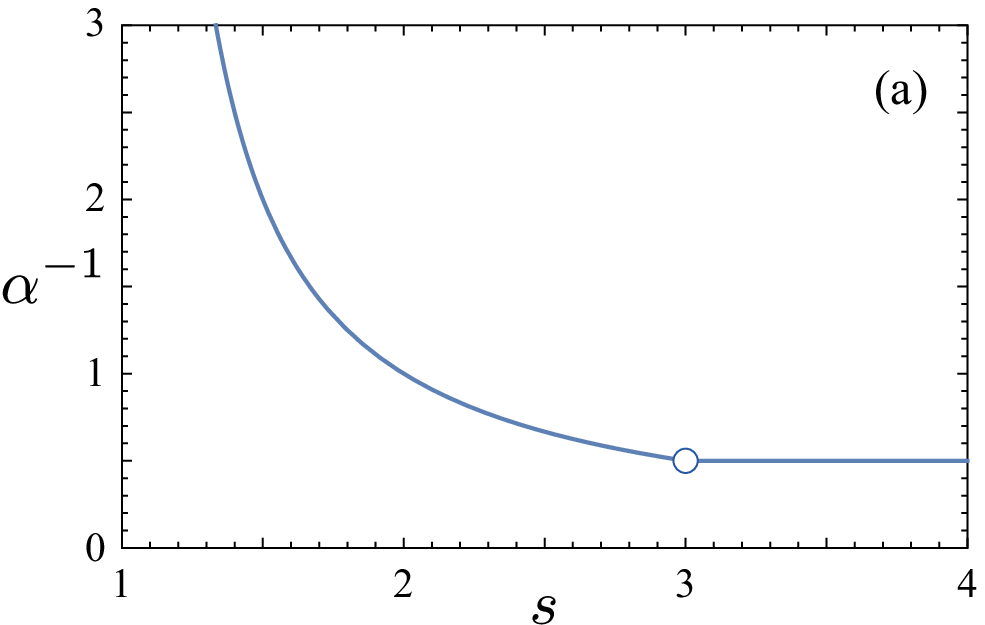}
\hfill
\includegraphics[width=0.48\textwidth]{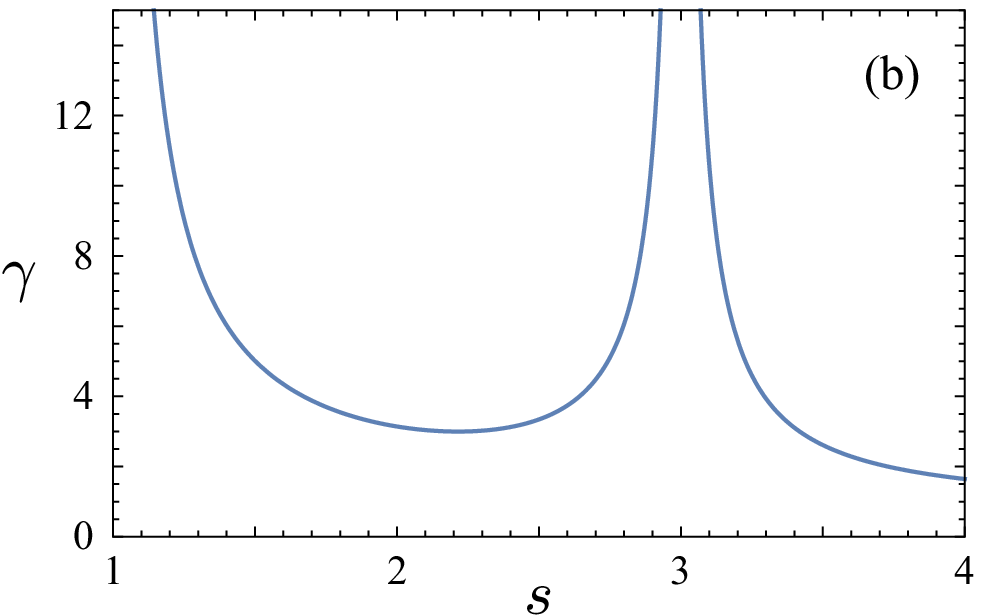}
\caption{The $s$-dependence of (a) $1/\alpha$ and (b) $\gamma$.}
\label{fig8}
\end{figure}

\begin{figure}
\centering
\begin{tabular}{cc}
\includegraphics[width=0.48\textwidth]{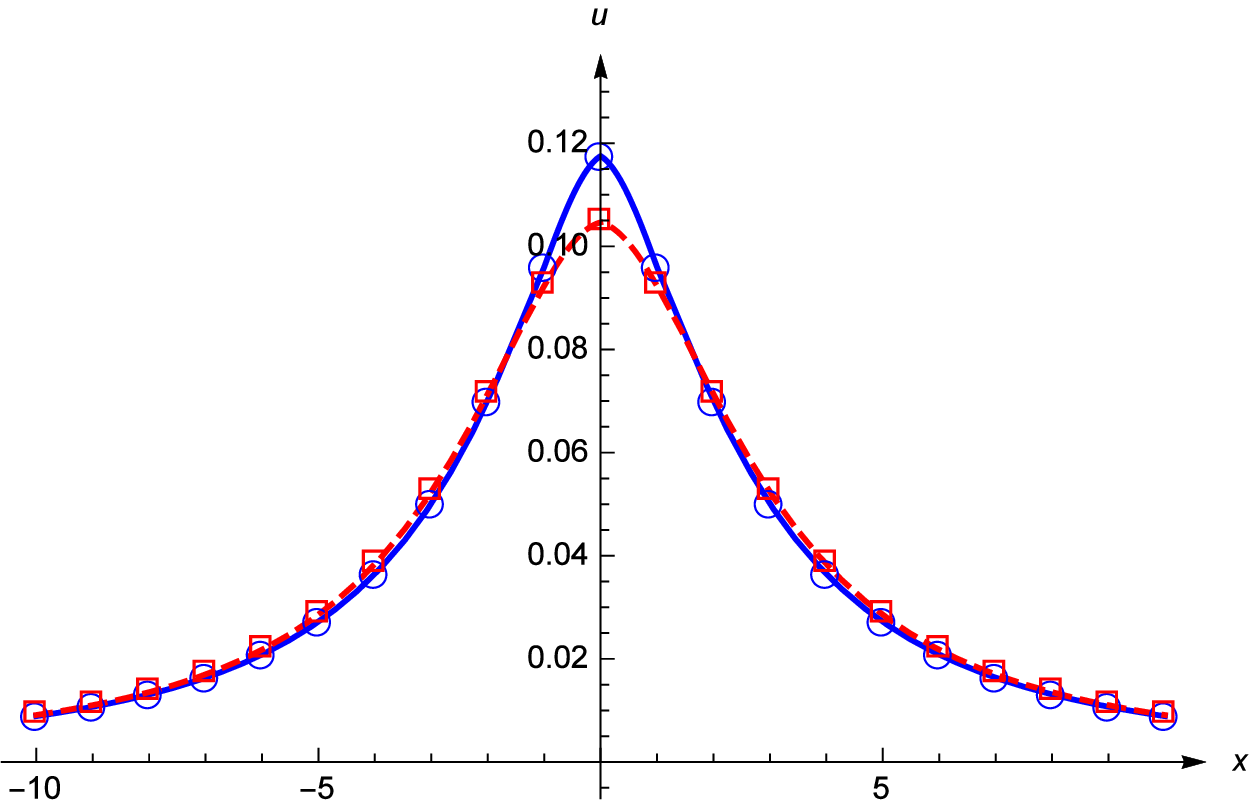}
&
\includegraphics[width=0.48\textwidth]{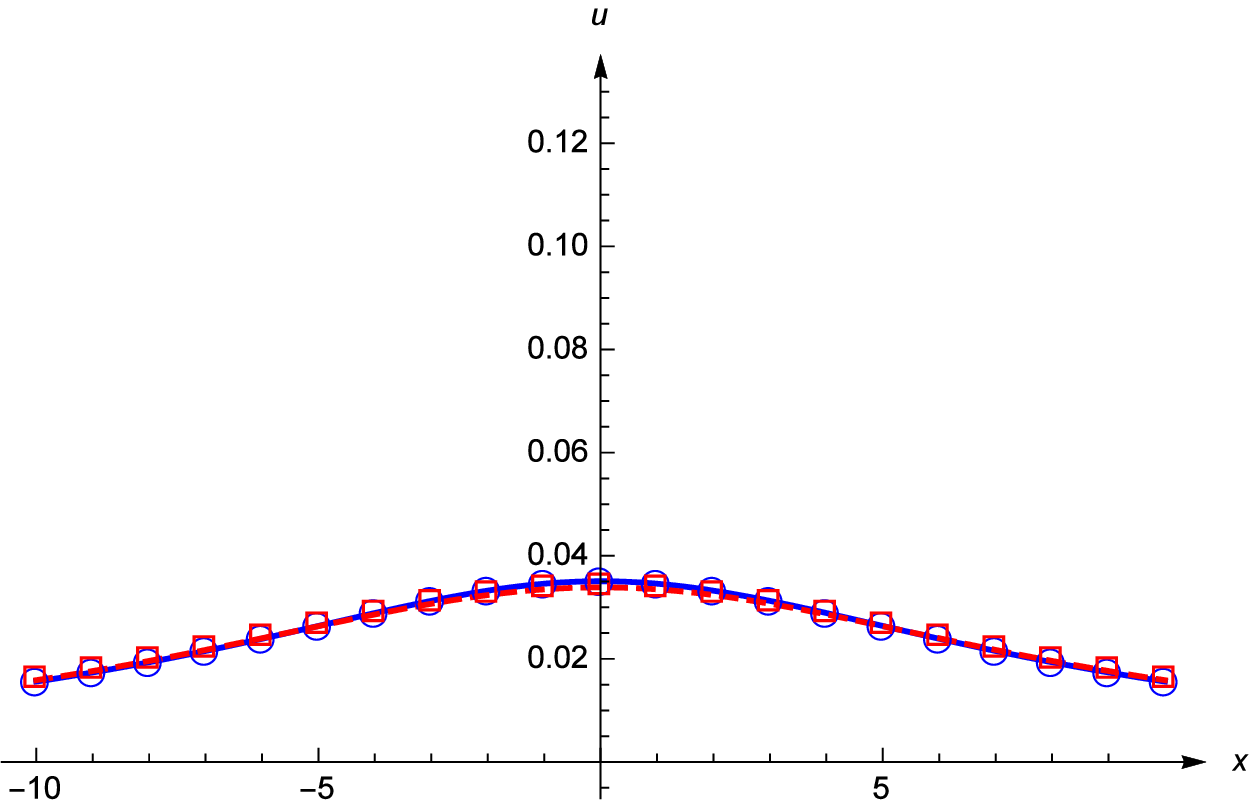}
\\
(a) & (b)
\end{tabular}
\caption{The solutions of \eqref{evol_equ}, \eqref{evol_equ_ic}
for $L=\tLM{s}$ (blue solid line with circles)
and for $L=cL_1^a$ (red dashed line with squares)
for $t=1$ (a) and $t=3$ (b).}
\label{fig9}
\end{figure}


\end{document}